\documentclass[hidelinks]{amsart}

\usepackage{stmaryrd, verbatim}
\usepackage{amssymb}
\usepackage{mathrsfs}
\usepackage{cancel}
\usepackage{wrapfig}
\usepackage{tikz}
\usepackage{xy}
\usepackage{hyperref}
\usepackage{cleveref}

\usetikzlibrary{calc,fadings,decorations.pathreplacing}

\xyoption{arrow}
\xyoption{curve}
\xyoption{matrix}
\xyoption{cmtip}
\SelectTips{cm}{}

\mathchardef\varDelta="7101

\def\quickop#1{\expandafter\DeclareMathOperator\csname
#1\endcsname{#1}}
\quickop{id}\quickop{Id}\quickop{op}
\quickop{Hom}
\quickop{Ob}\quickop{diag}
\quickop{hocolim}
\quickop{Sd}\quickop{Ex}

\numberwithin{equation}{section}

\newtheorem{theorem}[equation]{Theorem}
\newtheorem*{theorem*}{Theorem}
\newtheorem{corollary}[equation]{Corollary}
\newtheorem{lemma}[equation]{Lemma}
\newtheorem{proposition}[equation]{Proposition}
\theoremstyle{definition}
\newtheorem{definition}[equation]{Definition}

\theoremstyle{remark}
\newtheorem{remark}[equation]{Remark}
\newtheorem{claim}[equation]{Claim}

\newdir{ >}{{}*!/-5pt/\dir{>}}

\begin{document}

\title
[Random cover on symmetric space]{Random $\epsilon$-cover on symmetric space}

\author{Somnath Chakraborty}
\address{Fakult\"at f\"ur Mathematik, Ruhr-Universit\"at Bochum, 44801 Bochum, 
Deutschland}
\email{\href{somnath.chakraborty@rub.de}{somnath.chakraborty@rub.de}}

\maketitle

\begin{abstract} 
A randomized scheme that succeeds with 
probability $1-\delta$ (for any $\delta>0$) has been devised to construct (1) an 
equidistributed $\epsilon$-cover of a compact Riemannian 
symmetric space $\mathbb M$ of dimension $d_{\mathbb M}$ 
and antipodal dimension $\bar{d}_{\mathbb M}$, and (2) an approximate 
$(\lambda_r,2)$-design, using $n
(\epsilon,\delta)$-many Haar-random isometries of 
$\mathbb M$, where \begin{displaymath}n
(\epsilon,\delta):=O_{\mathbb M}\left(d_{\mathbb M}\ln
\left(\frac 1\epsilon\right)+\log\left(\frac 1\delta\right)
\right)\,,\end{displaymath} and $\lambda_r$ is the $r$-th smallest 
eigenvalue of the Laplace-Beltrami operator on $\mathbb M$. 
The $\epsilon$-cover so-produced 
can be used to compute the integral of 1-Lipschitz functions within 
additive $\tilde O(\epsilon)$-error, as well as in comparing persistence homology computed 
from data cloud to that of a hypothetical data cloud sampled from the uniform measure.
\end{abstract}



\section{Introduction}
The purpose of this paper is to devise a randomized 
scheme that produces (with success probability 
at least $1-\delta$ for any $0<\delta\leq 1$) 
an $\epsilon$-cover in a $d_{\mathbb M}$-dimensional compact connected 
Riemannian symmetric space $(\mathbb M, 
\zeta_{\boldsymbol{o}})$ --- 
where $\mathbb M=\mathbb K/\mathbb H$ with $\mathbb K$ 
a compact connected semisimple Lie group 
and $\mathbb H\subseteq \mathbb K$ a closed subgroup, 
and $\zeta_{\boldsymbol{o}}$ a geodesic reflecting global isometry 
with respect to the origin $\boldsymbol{o}:=
\mathbb H{\boldsymbol{e}}\in \mathbb M$. 
The underlying mechanism is to pick an alphabet consisting of 
$k := {O}(d_{\mathbb M}\ln (1/\epsilon) + \ln (1/\delta))$ 
independent Haar-random isometries from $\mathbb K$, take 
all possible concatenations of length $\ell := 
{O}(d_{\mathbb M}\ln (1/\epsilon))$ from this 
set of isometries, and make them act on a fixed (but arbitrary) point 
from $\mathbb M$. This draws a strong parallelism to the 
theory of geometric random walk, since the idea of the construction is essentially 
that of analysing certain lazy random walk on a Schrier graph on the space 
$\mathbb M$ --- generated by a nonempty set of isometries --- 
and show that the walk mixes fast because the random Schrier graph 
essentially has an expander-like property.
We show that the set of points so produced is equidistributed 
at a scale of $\epsilon$. The $\epsilon$-cover so generated, requires 
much smaller number of random samples than the naive Monte Carlo 
or other similar methods of 
taking a large number of i.i.d random samples 
from $\mathbb M$; for example, \cite{MR2383768} requires 
$O(d_{\mathbb M}\epsilon^{-d_{\mathbb M}}\log\epsilon^{-1})$ many 
i.i.d random samples from $\mathbb M$ to generate an $\epsilon$-cover of 
a $d_{\mathbb M}$-dimensional compact Riemannian manifold 
$\mathbb M$. The $\log$-size of the $\epsilon$-cover so produced 
differs from the $\log$-size of a hypothetical $\epsilon$-net 
--- whose size is the volumetric lower-bound of $\epsilon$-covering 
number of $\mathbb M$ --- by a $\log\log$ factor.
Quantitatively, our result is most interesting when we let $\epsilon\rightarrow 0$
for symmetric spaces $\mathbb M$ of fixed dimension and antipodal dimension. 
We are motivated in part by the possibility of using 
equidistributed points for the integration of Lipschitz functions over $\mathbb M$. 
The $\epsilon$-cover we construct 
is equidistributed (with probability at least $1-\delta$) in 
the following sense: the normalized counting measure 
over the net is close to the Haar-induced probability measure 
on $\mathbb M$ in $1$-Wasserstein distance. This implies that the integral of 
every $1$-Lipschitz function on the $\mathbb M$ with respect to 
the normalized counting measure 
is ``$\epsilon$-close" to its integral with respect to 
Haar-induced probability measure. 
Moreover, such a cover is an explicit approximate 
$(\lambda,2)$-design on $\mathbb M$. When, we have 
oracle access to the pairwise Riemannian-distances in $\mathbb M$, 
such an $\epsilon$-cover immediately yields an efficient computation 
of the singular homologies of $\mathbb M$. 
Furthermore, the bound on the Wasserstein-1 distance ensures 
small Prokhorov distance between the push-forwards along 
the persistence homology maps of the canonical 
Haar-induced probability measure in $\mathbb M$ and the empirical measure 
supported on the $\epsilon$-cover constructed in this paper, showing that 
the persistence homologies can be computed --- to within small bottleneck distance 
with high probability --- using point-cloud produced by such a 
construction.\\

We briefly survey earlier work relevant to the theme of this paper. 
It follows from theory of random walk on a Schrier graph 
$\chi(\mathbb M,\mathcal S)$ of compact symmetric 
space $\mathbb M$ --- with respect to a finite set $\mathcal S
\subseteq \mathbb K$ of its isometries --- that, a sufficient 
condition for rapid mixing of this random walk to the uniform 
distribution on $\mathbb M$ is that the averaging 
operator $z_\mu:\operatorname{L}_0^2(\mathbb M)\rightarrow 
\operatorname{L}_0^2(\mathbb M)$, given by 
\begin{equation}\label{aveop}z_\mu(f)(\boldsymbol{x})=\int_{\mathbb K}
f(\boldsymbol{x}\boldsymbol{g})~d\mu_{\mathbb K}(\boldsymbol{g})\end{equation}
posses a spectral gap; here $\mu_{\mathbb K}$ 
is the normalized empirical measure supported on the 
group of isometries underlying the Schrier graph. The existence 
of spectral gap for these operators looks far-fetched at this state, 
and we approach the problem via proving a weaker but sufficient 
version of the spectral gap phenomenon, as sketched below.\\

In \cite{MR1262979}, Alon and Roichman proved that, given 
any $\delta>0$, there exists a $c(\delta)>0$ such that for any 
finite group $G$, and a random subset $S\subset G$ of order 
at least $c(\delta)\log|G|$, the induced Cayley graph $\chi(G,S)$ 
has small normalized second largest eigenvalue (in absolute value), in that  
\begin{equation}\label{eqnD}{\mathbb E}\begin{pmatrix}{|\lambda_2^\ast
(\chi(G,S))|}\end{pmatrix}<\delta\end{equation} holds. Considering 
random walk on expander multigraphs, it follows that every element 
$g\in G$ is an $S$-word of length at most $O_\delta(\log |G|)$. 
Landau and Russell in \cite{MR2097328} deviced a short proof (with slightly better 
constants) of this result while rephrasing the question using representation 
theory. For an irreducible representation $\rho\in {\hat G}$, let $d_\rho$ 
be its dimension; let $R$ be the regular representation of $G$, and 
$D=\sum_{\rho\in{\hat G}}d_\rho$. Landau and Russell 
proved that \cref{eqnD} holds for random subsets 
$S\subset G$ of order at least 
\[\begin{pmatrix}{{\frac{2\ln 2}{\delta}}+o(1)}\end{pmatrix}^2\log |D|\] This 
was obtained via an application of \emph{tail bounds for operator-valued 
random variables}, as in Ahlswede and Winter \cite{MR1889969}, 
building upon the following observation: the normalized 
adjacency matrix of $\chi(G,S)$ is the operator \begin{equation}\label{op1}(2|S|)^{-1}
\sum_{\boldsymbol{s}\in S}(R(s)+R(s^{-1})),\end{equation} 
presented in terms of the standard basis of ${\mathbb C}[G]$. 
In equivalent terms, the quoted result from \cite{MR1262979} implies 
that random Cayley graph $\chi(G,S)$ is an expander, and also, the operator 
in \cref{op1} has a spectral gap satisfying \cref{eqnD}. 
When $G$ is a compact connected simple Lie group, and 
$\mu_G$ a left-invariantBorel probability measure on $G$, Benoist and de 
Saxc\'{e} in \cite{MR3529116} showed --- following earlier 
works \cite{MR2652483} by Bourgain and Gamburd in 
$G=SU_d$ case --- that spectral gap of the operator in 
\cref{aveop} is equivalent to $\mu$ being almost diophantine, a property 
known to be true when the support of $\mu$ is sufficiently 
(inexplicit) large, and consists of algebraic elements. However, 
note that these qualitative results are mostly not applicable 
in the randomized computational 
setting, since 1. the set of elements with algebraic entries is a 
measure zero subset of the Lie group, 2. the statement on the 
required support size is qualitative, and 3. checking almost diophantine 
property is not computationally feasible at this stage. 
In \cite{MR4138637}, quantitative version of the spectral 
gap question for compact Lie group $G$ was considered. 
It was shown that --- with high probability --- the Hausdorff 
distance between $G$ and the subset of fixed length words 
on a random finite alphabet $S\subset G$ decays exponentially 
fast with respect to the length of the words.\\

In this work, we work in the setting analogous to 
that of expander Cayley graphs on finite groups, as in 
\cite{MR2097328}. In summary, our analysis underlying the construction of 
the $\epsilon$-cover evolves around orthogonally projecting the heat kernel 
of $\mathbb M$ --- induced by the Casimir operator of $\mathbb K$ --- 
onto an appropriate finite-dimensional subspace of 
$\operatorname{L}^2(\mathbb M)$. 
In \cite{MR4356249}, this idea was used to generate equidistributed random 
$\epsilon$-cover on unit sphere $\mathbb S^d$ --- a rank one compact 
symmetric space, and the present work contains an extension of 
the results in \cite{MR4356249} to compact symmetric space of 
arbitrary rank. 
In the special case of the unitary group $U_d$, a similar result with a 
quadratic dependence on dimension for the length of the words 
was previously obtained by Hastings and Harrow, in \cite[theorem 5]{MR2553116}; 
however, in their result, the number of generators is specified in a indirect manner, 
whose dependence on the parameters of $U_d$ is not explicit. This present work 
yields the corresponding result for $\mathbb M=U_d$ (and the analogous 
results for any compact Lie groups) as a special case, 
and the bounds obtained here are shown to be 
linear in $d$, both for the number of generators and the length of the words. In fact, for these 
parameters, the value of $(2k)^\ell$ is close to the volumetric lower bound of 
$(1/\epsilon)^{\Omega(d)}$ on the size of an $\epsilon$-net of $\mathbb M$. 
Besides, the random $\epsilon$-net --- that comes out of our 
procedure --- has a ``small" description, from a computer science perspective: 
$viz$, it consists of the description of the random subset ${\mathcal S}\subseteq \mathbb K$, 
and of the maximum word length $\ell$. And, the number of random isometry 
required for the construction depends only on the parameters of the symmetric 
space itself, and not the group of isometries $\mathbb K$. 
We show that such an $\epsilon$-cover computes integral of 1-Lipschitz 
functions on $\mathbb M$, and the singular homology groups of $\mathbb M$ 
when given oracle access to the pairwise Riemannian-distances. Moreover, 
for any $n>0$, persistent homology of an random $n$-samples from the $\epsilon$-cover 
${\mathcal S}$ is $\sqrt{\epsilon}$-close to the persistence homology 
computed from $\sigma^{\mathbb M}$-uniform random $n$-samples 
(\cref{persistent2}), showing that the data cloud so constructed can be 
used --- in persistent homology landscape --- 
as a proxy for hypothetical data cloud sampled uniformly from $\mathbb M$.

\subsection{Main results}
\label{sec:main}



We now state our main results, as 
\cref{theorem:bigtheorem1,theorem:bigtheorem2,theorem:bigtheorem3}.

Throughout this subsection, $\mathbb M=\mathbb K/\mathbb H$ 
is a compact connected Riemannian symmetric space of dimension 
$d_{\mathbb M}$ and antipodal dimension $\bar{d}_{\mathbb M}$, 
with $\sigma^{\mathbb M}$ the probability measure on $\mathbb M$ 
corresponding to Haar probability measure on the compact connected 
semisimple Lie group $\mathbb K$.

\begin{theorem}\label{theorem:bigtheorem1}
There is a constant $C_{\mathbb M}>0$ --- that depends only on 
$\mathbb M$ --- such that for any $\delta\in (0,\frac12)$, each of the following 
statements hold with probability at least $1-2\delta$.

Let $\epsilon\in (0,2^{-e})$, and assume that \begin{equation*}r_{\epsilon,\mathbb M}=
2\epsilon\sqrt{\ln {\frac {3C_{\mathbb M}}{\epsilon^{2d_{\mathbb M}-
\bar{d}_{\mathbb M}-1}}}}\end{equation*} is small enough. Let 
$\mathcal S\subset \mathbb K$ be a random multiset consisting of iid random points, 
drawn from the Haar measure on $\mathbb K$, such that 
\begin{displaymath}|{\mathcal S}|\geq 16\ln 2\left(2\ln C_{\mathbb M}+\ln \frac 6\delta+\ln 
\left(\Gamma\left(\frac {d_{\mathbb M}}2+1\right)\right) +\frac{d_{\mathbb M}}2
\ln \frac 1{\pi\epsilon}+\ln \frac 1{{\mathfrak v}_{\mathbb M}}\right)
\end{displaymath} with $a_d:={\frac {2\log_2\log_2(5d)}
{\log_2(5d)}}$. Let $\boldsymbol{p}\in \mathbb M$ and define 
${\hat{\mathcal S}} := {\mathcal S} \cup {\mathcal S}^{-1}$. Suppose that 
\begin{displaymath}\ell\geq {\frac {d_{\mathbb M}}2}
\log_2\begin{pmatrix}{\frac 1{r\epsilon}}\end{pmatrix}+
(4+3a_d)d_{\mathbb M}\log_2 \begin{pmatrix}{\frac 1{\epsilon}}
\end{pmatrix}\,.\end{displaymath} Then 
\begin{enumerate}
\item \emph{($\epsilon$-cover on $\mathbb M$):} 
${\hat {\mathcal S}}^\ell\boldsymbol{p}\subseteq \mathbb M$ is an 
$r_{\epsilon,\mathbb M}$-cover of $\mathbb M$ for any $\boldsymbol{p}\in\mathbb M$;
\item \emph{(equidistribution of the $\epsilon$-cover):} the empirical 
measure $\nu$ on ${\hat {\mathcal S}}^\ell\boldsymbol{p}$ satisfies 
\begin{equation*}\operatorname{W}_1(\sigma^{\mathbb M},\nu)\leq 2
\sqrt{d_{\mathbb M}}\epsilon\,;\end{equation*}
\end{enumerate}
\end{theorem}
These are further discussed in details in \cref{ssec:1} and \cref{ssec:2} below (see 
\cref{mainsym1} and \cref{Wassers}).

\begin{theorem}\label{theorem:bigtheorem2}
Let $C_{\mathbb M}$ be as in \cref{theorem:bigtheorem1}. 
Let ${\mathscr S}:=\hat{\mathcal S}^\ell\boldsymbol{o}$ --- 
where ${\mathcal S}\subseteq\mathbb M$ is a random multisubset of isometries 
selected independently from the Haar measure on $\mathbb K$. 
Let 
$\delta\in (0,\frac 12)$. For any $\upsilon\in (0,1)$, 
and any integer $r>0$, if \begin{equation*}|{{\mathcal S}}|
=16\ln 2\ln\begin{pmatrix}\frac{C_{\mathbb M}\lambda_r^{\frac {d_{\mathbb M}}2}}\delta\end{pmatrix}\,,\end{equation*} 
and \begin{align*}\ell\geq \log_2\frac 1\upsilon+\log_2
C_{\mathbb M}+\frac {d_{\mathbb M}}4\log_2\lambda_r+\begin{pmatrix}d_{\mathbb M}-
\frac {\bar{d}_{\mathbb M}+1}2\end{pmatrix}\log_2\frac 1\epsilon\,,\end{align*} 
where $\epsilon= \lambda_r^{-\frac{d_{\mathbb M}+2}2}C_{\mathbb M}^{-\frac 14}
\upsilon^{\frac 12}$, then ${\mathscr S}\subseteq\mathbb M$ 
is an $\upsilon$-approximate $(\lambda_r,2)$-design with 
probability at least $1-2\delta$.
\end{theorem}

This is further discussed in details in \cref{ssec:3} below (see 
\cref{iterate3.1.1}).

\begin{theorem}\label{theorem:bigtheorem3}
Suppose that the random subsets $\hat{\mathcal S}\subseteq\mathbb K$ as well as the integer $\ell$ are as ststed in \cref{theorem:bigtheorem1}; then, for any $\boldsymbol{p}_0\in\mathbb M$, letting $\mathscr S_\ell=\hat{\mathcal S}^\ell\boldsymbol{p}_0$, the following inequality holds for all integers $q\geq 0$, with probability at least $1-\delta$: \begin{equation}\frac 1n \cdot d_{\operatorname{Pr}}\left(\Phi^{q,n}_{{\mathbb M},\partial_{\mathbb M},\sigma^{\mathbb M}}, \Phi^{q,n}_{{\mathscr S_{\ell}},\partial_{\mathbb M},\sigma^{\mathbb M}_\ell}\right)\leq \sqrt[4]{2d_{\mathbb M}\epsilon}\end{equation}
\end{theorem}

This is further discussed in details in \cref{ssec:4} below (see 
\cref{persistent1} and \cref{persistent2}).\\

To elaborate a bit more, suppose that $\mathbb K$ is one of the compact connected 
matrix groups such as $SU_d$ --- the group of orthogonal matrices having 
determinant one. We assume given a real number model of computation, in which 
only standard algebraic operations are allowed on random vectors, 
but bits are not manipulated. Now, for $\mathcal S$ and $\ell$ 
as described in \cref{theorem:bigtheorem1}, 
we consider all compositions of length $\ell$ of the elements in $\mathcal S$ 
and their inverses. Apply the resulting 
matrices to a vector $\boldsymbol{o}=\mathbb H\boldsymbol{e}$. Then these points 
form an equidistributed cover, that can be used for integrating any $1$-Lipschitz 
function on $\mathbb M$ to within an additive error of $\epsilon$. The persistence 
homology computed from such a point cloud well-approximates the one computed 
from point cloud formed out of random samples from $\mathbb M$. 
In the special case of an $d$-dimensional Euclidean sphere, 
if we assume an oracle that outputs independent $d$-dimensional random 
matrices from $\mathbb K$ when queried, then the whole process requires only $k$ 
queries to this oracle. Note that the obvious procedure of producing 
an equidistributed cover would require $\epsilon^{-\Omega(d)}$ calls to the 
oracle. The latter method uses exponentially more randomness than our 
procedure using random isometries.

\bigskip

\section{Preliminaries}
Let $\mathbb M$ be a connected compact (real) Riemannian symmetric 
space of dimension $d_{\mathbb M}$ and antipodal dimension $\bar{d}_{\mathbb M}$. 
If $\mathbb K$ is the identity component of the 
group of isometries of $\mathbb M$, then $\mathbb K$ is a compact 
Lie group that acts transitively 
on $\mathbb M$, leading to the identification $\mathbb M=\mathbb K/
\mathbb H$ for a closed subgroup $\mathbb H\subseteq \mathbb K$ --- 
the isotopy group of a point in $\mathbb M$. Let $\theta$ be the Cartan 
involution associated to the symmetric structure of $\mathbb M$. 
Let $d_{\mathbb K}$ be the dimension of $\mathbb K$, so that $d_{\mathbb H}
=d_{\mathbb K}-d_{\mathbb M}$. Let $\sigma^{\mathbb M}$ be the canonical 
probability measure on $\mathbb M$, induced by the Haar probability 
measure $\sigma^{\mathbb K}$ on $\mathbb K$, and let $\mu_{\mathbb M}$ be the 
canonical $\mathbb K$-invariant Riemannian top form on $\mathbb M$; 
there is a constant ${\mathfrak v}_{\mathbb M}>0$ such that 
$\sigma^{\mathbb M}={\mathfrak v}_{\mathbb M}
\mu_{\mathbb M}$. Recall that, for any measurable 
${\mathcal S}\subseteq \mathbb K$, one has 
\begin{displaymath}\sigma^{\mathbb M}(\mathcal S)
=\int_{\mathbb M}\mathbb I_{\mathcal S}(\boldsymbol{x})~
d\sigma^{\mathbb M}(\boldsymbol{x}) 
=\int_{\mathbb K}\mathbb I_{\mathbb H\mathcal S}(\boldsymbol{x})~
d\sigma^{\mathbb K}(\boldsymbol{x})\end{displaymath} The $\operatorname{L}^2$-norm of 
$\mathbb C$-valued square-integrable functions on $\mathbb M$ will 
always be with respect to the probability measure $\sigma^{\mathbb M}$.

\medskip

\subsection{Some algebraic prelimiaries}
From here on, we assume --- for the rest of this paper ---
that $\mathbb K$ is a semisimple compact connected 
separable Lie group, and $\mathbb H$ a closed subgroup. Note that, 
$\mathbb K$ has finite center by semisimple criterion. 
Consider the Cartan decomposition $\mathfrak k=\mathfrak m
\oplus \mathfrak h$ associated to the symmetric space $\mathbb M$; here, 
$\mathfrak k$ and $\mathfrak h$ are the Lie algebras of $\mathbb K$ and 
$\mathbb H$, respectively, and $\mathfrak m$ the tangent space 
to $\mathbb M$ at the origin $\boldsymbol{o}:={\mathbb H}\boldsymbol{e}$. 
Realising the Lie algebra $\mathfrak k$ alternatively 
as the algebra of Killing vector-fields on $\mathbb M$, 
one has \begin{displaymath}{\mathfrak h}:=\{{\mathfrak p}\in {\mathfrak g}: 
{\mathfrak p}_{\boldsymbol{o}}=0\}\,,\hspace{0.5cm}{\mathfrak m}:=
\{{\mathfrak p}\in {\mathfrak g}: (\nabla{\mathfrak p})_{\boldsymbol{o}}
=0\}\end{displaymath} Let $\operatorname{B}_{\mathbb K}$ 
denote the Killing form on $\mathfrak k$; thus, 
\begin{equation}\label{killing-form}\operatorname{B}_{\mathbb K}
(\mathfrak p_1,\mathfrak p_2):=\operatorname{tr}(
\operatorname{ad}_{\mathfrak p_1}\circ \operatorname{ad}_{\mathfrak p_1})
\end{equation} We denote the norm on $\mathfrak k$ induced by 
$\operatorname{B}_{\mathbb K}$ simply by $|\!|\cdot|\!|$. 
Note that, since $\operatorname{ad}:{\mathfrak k}\rightarrow 
{\mathfrak g}{\mathfrak l}({\mathfrak k})$ is a Lie-algebra morphism, the following 
holds for any ${\mathfrak p}_1,{\mathfrak p}_2,
{\mathfrak p}\in {\mathfrak k}$: \begin{align*}
\operatorname{B}_{\mathbb K}(\operatorname{ad}_{\mathfrak p}({\mathfrak p}_1),{\mathfrak p}_2)&=
\operatorname{B}_{\mathbb K}({\mathfrak p}_1, \operatorname{ad}_{\mathfrak p}
({\mathfrak p}_2))\,.\end{align*} 
By \emph{Cartan's criterion} for semisimplicity (\cite{MR499562}), 
$\mathbb K$ being semi-simple implies 
$\operatorname{B}_{\mathbb K}$ is nondegenerate. The metric on 
$\mathbb M$ is induced by the Killing form $\operatorname{B}_{\mathbb K}$
on the Lie algebra $\mathfrak k$ of $\mathbb K$, via the 
decomposition $\mathfrak k=\mathfrak m\oplus \mathfrak h$.  
Let $\nabla_{\mathbb M}$ denote the canonical Levi-Civita connection 
on $\mathbb M$, with respect to this metric. 
A curve $\gamma$ in $\mathbb M$ is a geodesic 
if and only if $\nabla_{\mathbb M}(\dot{\gamma},\dot{\gamma})
=0$ holds. Recall (\cite{MR103509}) 
that the geodesics through the origin in $\mathbb M$ are precisely the 
images under the canonical projection $\pi:\mathbb K\rightarrow \mathbb M$ 
of the $\mathbb H$-transversal geodesics in $\mathbb K$; since the geodesics 
in $\mathbb K$ are the 1-parameter subgroups in $\mathbb K$, 
the same holds for geodesic through origin of $\mathbb M$.\\

Since the compact Lie group $\mathbb K$ acts on 
the separable (\cite{MR0262773}) Hilbert-space 
$\operatorname{L}^2(\mathbb M)$ --- via the unitary 
`regular' representation $(\boldsymbol{s}\cdot \phi)(\boldsymbol{x}):=
\phi(\boldsymbol{x}\boldsymbol{s})$, the machinery of Peter-Weyl theorem
(\cite{MR1410059}) implies that there is 
Hilbert-space decomposition of 
$\operatorname{L}^2(\mathbb M)$ as orthogonal direct 
sum of unitary irreducible unitary $\mathbb K$-representations; that is, 
\begin{equation}\operatorname{L}^2(\mathbb M)
\cong\displaystyle\bigoplus_{n=1}^\infty V_{\pi_n}\,,
\end{equation} where the direct sum is over a subset of equivalence classes $[\pi_n]$ of 
irreducible $\mathbb K$-representations $(V_{\pi_n},\pi_n)$. Note that, 
by compactness of the Lie group $\mathbb K$, the separable Hilbert space 
$\operatorname{L}^2(\mathbb K)$ --- having an orthonormal basis of 
matrix coefficients of all the unitary irreducible representations,  
which are all finite-dimensional (Peter-Weyl theorem); 
in particular, there are only countably many inequivalent 
irreducible representations of $\mathbb K$. Let ${\mathcal D}
(\mathbb M)$ denote the algebra of smooth functions on $\mathbb M$. 
This is a sub-representation of the unitary regular representation of 
$\mathbb K$ on $\operatorname{L}^2(\mathbb M)$.

\begin{proposition}[Helgason, \cite{MR1790156}]\label{helgason1}
Each $V_{\pi_n}$ is a joint eigenspace, in ${\mathcal D}(\mathbb M)$, 
of the algebra of $\mathbb K$-invariant differential operators on $\mathbb M$. 
Moreover, each irreducible $\mathbb K$-representation contained in 
${\mathcal D}(\mathbb M)\subseteq\operatorname{L}^2(\mathbb M)$ 
arise as a direct summand $V_{\pi_n}$ with multiplicity one.
\end{proposition}

For the purpose of this paper, we will need to realise a somewhat 
more concrete version of the orthonormal decomposition of 
$\operatorname{L}^2(\mathbb M)$ into finite dimensional 
$\mathbb K$-invariant subspaces. 
To this end, we present a brief exposition of 
the notion of Casimir operator from representation theory of 
compact Lie groups. Let $\Delta_{\mathbb M}$ denote the 
Laplace-Beltrami operator with respect to the Riemannian 
metric on $\mathbb M$. Let ${\mathscr O}(\mathbb M)$ denote 
the algebra of 
invariant differential functions on $\mathbb M$. Note that, a 
differential operator $D$ on $\mathbb M$ is in ${\mathscr O}(\mathbb M)$ 
if \begin{equation}\label{0inv0}D(\Phi^{\boldsymbol{s}^{-1}})=D(\Phi)^{\boldsymbol{s}^{-1}}
\end{equation} holds for all $\boldsymbol{s}\in\mathbb K$ and 
$\Phi\in {\mathcal D}(\mathbb M)$; here $f^{\boldsymbol{s}^{-1}}(\boldsymbol{x})
:=f(\boldsymbol{s}\boldsymbol{x})$ for all $\boldsymbol{x}\in\mathbb M$. 
The canonical projection $\pi:\mathbb K\rightarrow \mathbb M$ associated to the 
(normal) homogeneous structure of $\mathbb M=\mathbb K/\mathbb H$ 
implies canonical indentification of complex Hilbert spaces with $\mathbb K$-actions: 
\begin{equation}\label{hilbert1}\operatorname{L}^2(\mathbb M)\cong 
\operatorname{L}^2(\mathbb K)^{\mathbb H}:=
\{\phi\in \operatorname{L}^2(\mathbb K): \phi^{\mathbb H}
=\phi\}\,.\end{equation} Moreover, this restricts to the space of smooth functions too: 
\begin{equation}\label{hilbert2}
{\mathcal D}(\mathbb M)\cong 
{\mathcal D}(\mathbb M)^{\mathbb H}:=
\{\phi\in {\mathcal D}(\mathbb M): \phi^{\mathbb H}
=\phi\}\,.\end{equation}



Let $\{{\mathfrak p}_a:a\in [d_{\mathbb K}]\}$ be 
a basis of $\mathfrak k$; let $\{{\mathfrak p}_a^\star:a\in [d_{\mathbb K}]\}$ be 
a Killing-dual basis of $\mathfrak k$. The Casimir element 
$\Omega_{\mathbb K}$ of $\mathbb K$ is defined by \begin{equation*}
\Omega_{\mathbb K}=\sum_{a=1}^{d_{\mathbb K}}
{\mathfrak p}_a\cdot {\mathfrak p}_a^\star\,.\end{equation*} 
A change-of-basis argument shows that the Casimir operator 
$\Omega_{\mathbb K}$ is basis-independent. Thus, we may take 
$\{{\mathfrak p}_a:a\in [d_{\mathbb K}]\}$ to be a Killing-orthonormal 
basis of $\mathfrak k$, compatible with the Cartan decomposition; 
then, the Casimir operator is \begin{equation}\label{casimir}
\Omega_{\mathbb K}=\sum_{a=1}^{d_{\mathbb K}}
{\mathfrak p}_a^2\,.\end{equation} 
Note that, for any $a,b\in [d_{\mathbb K}]
$, one has --- following Einstein's summation convention --- that 
$\operatorname{ad}_{{\mathfrak p}_b}
({\mathfrak p}_a)=
(s_{a,b}^c){\mathfrak p}_c$ 
for some scalars $s_{a,b}^c$: 
it follows that \begin{equation*}0=\operatorname{B}_{\mathbb K}(
\operatorname{ad}_{{\mathfrak p}_b}({\mathfrak p}_a), {\mathfrak p}_c)
+\operatorname{B}_{\mathbb K}({\mathfrak p}_a,
\operatorname{ad}_{{\mathfrak p}_b}({\mathfrak p}_c))=s_{a,b}^c
+s_{c,b}^a\,,\end{equation*} which implies 
$0=\operatorname{ad}_{{\mathfrak p}_b}(\Omega_{\mathbb K})$. 
Linearity yields $\operatorname{ad}_{\mathfrak p}
(\Omega_{\mathbb K})=0$ for every ${\mathfrak p}\in{\mathfrak k}$. 
In effect, this shows that $\Omega_{\mathbb K}$ lies in 
$\operatorname{Z}({\mathscr U}_{\mathfrak k})$ --- the center of 
the universal enveloping algebra ${\mathscr U}_{\mathfrak k}$.\\

To each ${\mathfrak p}\in {\mathfrak k}$ is associated a linear operator 
$\operatorname{D}_{\mathfrak p}:{\mathcal D}(\mathbb K)\rightarrow {\mathcal D}
(\mathbb K)$, defined by \begin{displaymath}\operatorname{D}_{\mathfrak p}
\Phi(\boldsymbol{x}):={\frac d{dt}}\bigg\rvert_{t=0}\Phi\left(\exp(t{\mathfrak p})
\cdot \boldsymbol{x}\right)\,.\end{displaymath} 
\begin{definition}\label{casimir-laplace-beltrami}
The Casimir operator $\operatorname{D}_{\mathbb K,\Omega_{\mathbb K}}: 
{\mathcal D}(\mathbb K)\rightarrow {\mathcal D}
(\mathbb K)$ is defined by \begin{equation}\label{clb}
\operatorname{D}_{\mathbb K,\Omega_{\mathbb K}}
:=\sum_{a=1}^{d_{\mathbb K}}\operatorname{D}_{{\mathfrak p}_a}^2
\,.\end{equation}
\end{definition}
Note that, by the identification in \eqref{hilbert2} above, 
$\operatorname{D}_{\mathbb K,\Omega_{\mathbb K}}$ 
restricts to a differential operator $\operatorname{D}_{\mathbb M,
\Omega_{\mathbb K}}$ --- the \emph{Casimir-Laplace-Beltrami} 
operator --- on ${\mathcal D}(\mathbb M)$. 
Now, for $\Phi\in {\mathcal D}(\mathbb K)$, it follows from the definition 
of $\operatorname{D}_{{\mathfrak p}_a}$ that 
\begin{displaymath}\operatorname{D}_{\mathbb K,
\Omega_{\mathbb K}}(\Phi)(\boldsymbol{x})=\sum_{a=1}^{d_{\mathbb K}}
{\frac {\partial}{\partial t}}\bigg\rvert_{t=0}{\frac {\partial}
{\partial t'}}\bigg\rvert_{t'=0}\Phi\left(\exp(t'{\mathfrak p}_a)\cdot 
\exp(t{\mathfrak p}_a)\cdot\boldsymbol{x}\right)\,.\end{displaymath} 

Let $\operatorname{D}_{\mathbb M,\Omega_{\mathbb K}}$ 
be the restriction of $\operatorname{D}_{\mathbb K,\Omega_{\mathbb K}}$ 
to ${\mathcal D}(\mathbb M)$. 
The following lemma is well-known; see \cite{MR1790156} for a proof.
\begin{lemma}\label{clb1}
$\Delta_{\mathbb M}=\operatorname{D}_{\mathbb M,\Omega_{\mathbb K}}$.
\end{lemma}


By a standard argument using Rellich-Kondrachov theorem (\cite{MR2597943}), 
it follows that the resolvent of the self-adjoint positive definite operator 
$-\operatorname{D}_{\mathbb M,\Omega_{\mathbb K}}=
-\Delta_{\mathbb M}$ --- which is defined on the 
dense subspace ${\mathcal D}(\mathbb M)$ in strong sense --- is compact, 
and therefore, $-\operatorname{D}_{\mathbb M,\Omega_{\mathbb K}}$ 
possess only a discrete spectrum. 
Let \begin{equation*}0=\lambda_0<\lambda_1<\lambda_2<\cdots\end{equation*} 
be the eigenvalues of $-\Delta_{\mathbb M}$. 
An application of spectral theorem implies the following orthogonal 
decomposition of $\operatorname{L}^2(\mathbb M)$ into the 
eigenspaces of $\operatorname{D}_{\mathbb M,\Omega_{\mathbb K}}$: 

\begin{theorem}[Spectral Theorem]\label{specthm}
Let $\mathscr E_{\mathbb M}=\{\lambda_0,\lambda_1,\cdots\}$ 
denote the spectrum of $-\Delta_{\mathbb M}$, and 
for each $\lambda\in\mathscr E_{\mathbb M}$, let ${\mathcal H}_\lambda(\mathbb M)
\subseteq \operatorname{L}^2(\mathbb M)$ denote the null-space 
of $\lambda \mathbb I+\Delta_{\mathbb M}$; then 
\begin{equation}\label{laplace-sum}\operatorname{L}^2(\mathbb M)=
\bigoplus_{\lambda\in \mathscr E_{\mathbb M}}{\mathcal H}_\lambda
(\mathbb M)\,.\end{equation} 
\end{theorem}

Note that, each eigenspace ${\mathcal H}_{\lambda}$ is a 
sub-representation of ${\mathcal D}(\mathbb M)$.
By \cref{helgason1} and irreducibility of $V_{\pi_n}$, 
for each $\lambda\in {\mathscr E}_{\mathbb M}$ there is $n_\lambda\in\mathbb Z_+$ 
such that $V_{\pi_{n_\lambda}}\subseteq {\mathcal H}_{\lambda}$. 
Let $\omega_{n_\lambda}$ be the highest integral weight corresponding 
to the irreducible representation $(V_{\pi_{n_\lambda}},\pi_{n_\lambda})$. An application of 
the Casimir-van der Waerden formula (\cite{MR3001808}), 
together with \cref{clb1}, implies \begin{equation}
\lambda:=|\!|\omega_{n_\lambda}+\rho|\!|^2
-|\!|\rho|\!|^2\end{equation} for all $n\in\mathbb Z_+$; 
here, $\rho$ denotes the half-sum 
of the positive weights of $\mathfrak k$. Note the trivial 
inequality \begin{equation}\label{big-omega}
|\!|\omega_{n_\lambda}|\!|^2\leq \lambda\leq 3|\!|\omega_{n_\lambda}|\!|^2\end{equation} 
for all $\omega_{n_\lambda}$ satisfying $|\!|\omega_{n_\lambda}|\!|\geq |\!|\rho|\!|$. 

\medskip



\subsection{Some estimates on the heat kernel on $\mathbb M$}
Going forward, we will need to apply the following result that 
bounds the eigenvalues and eigenspace-dimensions of 
the Laplace-Beltrami operator on a compact Riemannian manifold. 

Consider the differential operator \begin{equation}
\label{gen0}{\mathcal H}_{\mathbb M}:
=\frac d{dt}-\frac 12\Delta_{\mathbb M}\,.\end{equation} 
For any $\boldsymbol{p}\in\mathbb M$, the heat kernel 
$\operatorname{H}_{\boldsymbol{p}}(\boldsymbol{x},t)$ 
on $\mathbb M$ is the unique (smooth) solution to the following problem: \begin{equation}
\label{heateq0}{\mathcal H}_{\mathbb M}u
=0, \hspace{0.5cm}\lim_{t\rightarrow0^+}u(\boldsymbol{x}, t)=
\delta_{\boldsymbol{p}}(\boldsymbol{x})\,,\end{equation} 
where the limit above is taken in the distribution sense. 
Note that the following identity is immediate from 
$\mathbb K$-invariance of the Casimir operator 
and \cref{clb1}: \begin{equation}\forall~\boldsymbol{s}\in\mathbb K\,, 
\hspace{1cm}\operatorname{H}_{\boldsymbol{p}}(\boldsymbol{x},t)
=\operatorname{H}_{\boldsymbol{s}\boldsymbol{p}}
(\boldsymbol{s}\boldsymbol{x},t)\,.\end{equation}

\begin{proposition}[Donnelly, \cite{MR2213686}]\label{theorem:26-10-sc1} 
There is a constant $C_{\mathbb M}>0$, 
depending only on the lower bound of the injectivity radius of $\mathbb M$, (upper 
bound on the) sectional curvature of 
$\mathbb M$, the volume of $\mathbb M$, and dimension 
$d$, such that the following inequality holds
: 
\begin{equation}d_\lambda:=\dim {\mathcal H}_{\lambda}
\leq C_{\mathbb M}\lambda^{\frac{d-1}4}\,.\end{equation}
\end{proposition}

The Fourier decomposition (\cite{MR2569498}) of the heat kernel 
on $\mathbb M$ is \begin{align}\label{eq:heat2}\operatorname{H}_{\boldsymbol{p}}
(\boldsymbol{x},\epsilon^2)&= \sum_{\lambda\in
{\mathscr E}_{\mathbb M}} 
e^{-\lambda t}\sum_{j\in [d_{\lambda}]} 
\phi_{\lambda,j}(\boldsymbol{p})\phi_{\lambda,j}(\boldsymbol{x})
\end{align} for any fixed choice of an ordered orthonormal basis $(\phi_{\lambda,1},\dots,
\phi_{\lambda,\dim {\mathcal H}_\lambda})$ for each $\lambda\in
{\mathscr E}_{\mathbb M}$. For any $\lambda_\infty
>0$, let \begin{displaymath}\operatorname{H}_{\boldsymbol{p}}^{\lambda_\infty}
(\boldsymbol{x},\epsilon^2):= \sum_{\lambda\leq\lambda_\infty} e^{-\lambda t}
\sum_{j\in [d_\lambda]} \phi_{\lambda,j}
(\boldsymbol{p})\phi_{\lambda,j}(\boldsymbol{x})\end{displaymath} 
For integer $\lambda_\infty>0$, let ${\mathcal E}_{\lambda_\infty}(\mathbb M)
\subseteq \operatorname{L}^2(\mathbb M)$ denote 
the direct sum \[{\mathcal E}_{\lambda_\infty}(\mathbb M)
:=\bigoplus_{\lambda\in {\mathscr E}_{\mathbb M}(\lambda_\infty)}
{\mathcal H}_{\lambda}\,,\] where ${\mathscr E}_{\mathbb M}(\lambda_\infty):=
\{\lambda\in {\mathscr E}_{\mathbb M}:\lambda\leq\lambda_\infty\}$. 
Then, $\operatorname{H}_{\boldsymbol{p}}^{(\lambda)}
(\boldsymbol{x},\epsilon^2)$ is the orthogonal projection of 
$\operatorname{H}_{\boldsymbol{p}}(\boldsymbol{x},\epsilon^2)$ 
on ${\mathcal E}_\lambda(\mathbb M)$.\\

\begin{proposition}[Nowak et al, \cite{MR4333421}]\label{Nowak}
There are constants $C_{\mathbb M}>c_{\mathbb M}>0$ --- depending 
only on $d$ and $\bar{d}_{\mathbb M}$ --- such that 
the following holds for all $t\in (0,1)$, and all $\boldsymbol{x},\boldsymbol{p}\in 
\mathbb M$: \[\frac{c_{\mathbb M}\exp\left(-\frac{\partial_{\mathbb M}(
\boldsymbol{x},\boldsymbol{p})^2}{4t}\right)}{t_{\mathbb M, 
\boldsymbol{p}}^{\frac{d_{\mathbb M}-\bar{d}_{\mathbb M}-1}2}
t^{\frac {d_{\mathbb M}}2}}\leq\operatorname{H}_{\boldsymbol{p}}
(\boldsymbol{x},t)\leq \frac{C_{\mathbb M}\exp\left(-\frac{\partial_{\mathbb M}(
\boldsymbol{x},\boldsymbol{p})^2}{4t}\right)}{t_{\mathbb M, 
\boldsymbol{p}}^{\frac{d_{\mathbb M}-\bar{d}_{\mathbb M}-1}2}
t^{\frac {d_{\mathbb M}}2}}\,.\] Here $t_{\mathbb M, \boldsymbol{p}}:=
t+\operatorname{diam}(\mathbb M)-\partial_{\mathbb M}
(\boldsymbol{x},\boldsymbol{p})$.
\end{proposition}

It follows from this proposition that $\operatorname{H}_{\boldsymbol{p}}
(\boldsymbol{x},t)\geq 0$ for all points $\boldsymbol{x}, \boldsymbol{p}
\in\mathbb M$ and all $t>0$. Moreover, $\mathbb M$ being compact 
Riemannian symmetric space (hence, in particular, complete 
and geodesically complete, as in \cite{MR2569498}), one 
has \begin{equation}\label{stocom}
\int_{\mathbb M}\operatorname{H}_{\boldsymbol{p}}
(\boldsymbol{x},t)~d\sigma^{\mathbb M}({\boldsymbol{x}})
={\mathfrak v}_{\mathbb M}\,,\end{equation} a constant: in other words, 
$\mathbb M$ is stochastically complete.

\begin{remark}\label{heat-ineq-rem1}
For any $\epsilon,\eta\in (0,1)$, let \begin{equation}\label{r-eq}
r(\epsilon,\eta):=2\epsilon\sqrt{\ln\frac {3C_{\mathbb M}}
{\eta\epsilon^{2d_{\mathbb M}-\bar{d}_{\mathbb M}-1}}}\end{equation}
Fix ${\boldsymbol{p}}\in \mathbb M$; then, for any ${\boldsymbol{x}}\in
\mathbb M$ satisfying $\partial_{\mathbb M}(\boldsymbol{x},\boldsymbol{p})
\geq r(\epsilon,\eta)$, it follows from \cref{Nowak} that \begin{equation}\label{heat-ineq1}
0\leq \operatorname{H}_{\boldsymbol{p}}(\boldsymbol{x},\epsilon^2)
\leq \frac\eta3\,.\end{equation} Notice that, when 
$\eta=\Omega(\operatorname{poly}\epsilon)$, 
one has $r(\epsilon,\eta)=o(\epsilon^{1-{\mathfrak t}})$ for 
every constant ${\mathfrak t}>0$.
\end{remark}

The following statement was established in course of proving 
\textit{lemma 5} in \cite{MR630777}, although the 
dependence of the constant on manifold parameters were not explicit
; we note the statement as a lemma here, for easy reference, and to make 
the dependence of the constants on the manifold parameters explicit. Recall 
that the constant $C_{\mathbb M}>0$ appearing above 
depends only on $d$ and $\bar{d}_{\mathbb M}$.

\begin{lemma}\label{1des1}
Let $\mathbb M$ be a compact Riemannian symmetric 
space of dimension $d$. For 
$\epsilon\in (0,1)$, the following holds: \begin{align}\label{eq:heat69}
|\!|\operatorname{H}_{\boldsymbol{p}}(\boldsymbol{x},
\epsilon^2)|\!|_{\operatorname{L}^2(\mathbb M)}
\leq \sqrt{C_{\mathbb M}{\mathfrak v}_{\mathbb M}
2^{-d_{\mathbb M}+\frac{\bar{d}_{\mathbb M}+1}2}}
\epsilon^{-d_{\mathbb M}+\frac{\bar{d}_{\mathbb M}+1}2}
\end{align} 
\end{lemma}

\begin{proof}
By the upper-bound in \cref{Nowak}, it follows that 
\begin{align*}\operatorname{H}_{\boldsymbol{p}}(\boldsymbol{p},2\epsilon^2)&
\leq C_{\mathbb M}(2\epsilon^2)^{-d_{\mathbb M}+\frac{\bar{d}_{\mathbb M}+1}2}\end{align*} holds for 
$\epsilon\in (0,1)$. Now, by semigroup property of the heat kernel, one has 
\begin{align*}\int_{\mathbb M}\operatorname{H}_{\boldsymbol{p}}
(\boldsymbol{x},\epsilon^2)^2~d\sigma^{\mathbb M}(
\boldsymbol{x})&={\mathfrak v}_{\mathbb M}
\operatorname{H}_{\boldsymbol{p}}(\boldsymbol{p},2\epsilon^2)\\
&\leq C_{\mathbb M}{\mathfrak v}_{\mathbb M}
(2\epsilon^2)^{-d_{\mathbb M}+\frac{\bar{d}_{\mathbb M}+1}2}\,,\end{align*} which 
establishes the lemma. 
\end{proof}

\begin{corollary}\label{trunc}
For any $\lambda_\infty>0$, the following inequality holds for 
all $\epsilon\in (0,1)$: \begin{align}\label{eq:heat60}
|\!|\operatorname{H}^{(\lambda_\infty)}_{\boldsymbol{p}}(\boldsymbol{x},
\epsilon^2)|\!|_{\operatorname{L}^2(\mathbb M)}\leq 
\sqrt{C_{\mathbb M}{\mathfrak v}_{\mathbb M}}
\epsilon^{-d_{\mathbb M}+\frac{\bar{d}_{\mathbb M}+1}2}\end{align}
\end{corollary}

\begin{proof}
Immediate, since \begin{displaymath}|\!|\operatorname{H}^{(\lambda_\infty)}_{\boldsymbol{p}}(\boldsymbol{x},
\epsilon^2)|\!|_{\operatorname{L}^2(\mathbb M)}\leq |\!|\operatorname{H}_{\boldsymbol{p}}(\boldsymbol{x},
\epsilon^2)|\!|_{\operatorname{L}^2(\mathbb M)}\end{displaymath} by property of 
orthogonal projection on subspaces of Hilbert space.
\end{proof}

Going forward, we make essential use of Weyl's asymptotic 
of the eigenvalues of $\Delta_{\mathbb M}$ in a Riemannian manifold 
$\mathbb M$, as formulated in \cite{MR405514}.

\begin{lemma}[Weyl]
The following holds: \begin{equation*}\lim_{\lambda_\infty\rightarrow\infty}
\lambda_\infty^{-\frac {d_{\mathbb M}}2}\sum_{\lambda\in {\mathscr E}_{\mathbb M}
(\lambda_\infty)}d_\lambda = C_{\mathbb M}
\end{equation*}
\end{lemma}

From Weyl's law, it follows that \begin{equation*}
\lim_{\lambda_\infty\rightarrow\infty}\left(\lambda_\infty^{-\frac {d_{\mathbb M}}2}\sum_{\lambda
\in {\mathscr E}_{\mathbb M}(4^{\frac 1{d_{\mathbb M}}}\lambda_\infty)}d_\lambda\right)-
\lim_{\lambda_\infty\rightarrow\infty}\left(\lambda_\infty^{-\frac {d_{\mathbb M}}2}\sum_{\lambda
\in {\mathscr E}_{\mathbb M}(\lambda_\infty)}d_\lambda\right)= C_{\mathbb M}\,,
\end{equation*} and in particular, \begin{equation}\label{Weyls}
\sum_{\lambda\in {\mathscr E}_{\mathbb M}(\lambda_\infty,4^{\frac 1d}\lambda_\infty)}
d_\lambda\leq C_{\mathbb M}\lambda_\infty^{\frac {d_{\mathbb M}}2}\end{equation} for 
all $\lambda_\infty>0$, where ${\mathscr E}_{\mathbb M}(\lambda_\infty,4^{\frac 1{d_{\mathbb M}}}
\lambda_\infty):={\mathscr E}_{\mathbb M}(4^{\frac 1{d_{\mathbb M}}}\lambda_\infty)\setminus 
{\mathscr E}_{\mathbb M}(\lambda_\infty)$.

\bigskip

\section{Truncation of heat kernel}
In this section, we show that the $\operatorname{L}^2$-mass 
of the heat kernel can be well-approximated by the 
$\operatorname{L}^2$-mass of its orthogonal projection 
onto the subspace ${\mathscr E}_{\mathbb M}(\lambda_\infty)
\subseteq \operatorname{L}^2(\mathbb M)$ for a suitable $\lambda_\infty$ 
large enough. We start with an auxillery lemma.

\begin{lemma}\label{el1}
For any $\Delta_{\mathbb M}$-eigenvalue $-\lambda$, and 
orthonormal basis $\phi_{\lambda,1},\cdots,\phi_{\lambda,d_{\lambda}}$ 
of ${\mathcal H}_\lambda$, let $\psi_\lambda:\mathbb M^2\rightarrow
\mathbb C$ be defined by \begin{equation*}
\psi_\lambda(\boldsymbol{x},\boldsymbol{y}):=\sum_{j=1}^{d_\lambda}
\phi_{\lambda,j}(\boldsymbol{x})\overline{\phi_{\lambda,j}(\boldsymbol{y})}\,.
\end{equation*} Then the following statements hold: \begin{itemize}
\item $\psi_\lambda$ is independent of the orthonormal basis;
\item $\psi_\lambda\circ(\zeta,\zeta)=\psi_\lambda$ for any $\zeta\in\mathbb K$; 
\item $\psi_\lambda(\boldsymbol{p}, \boldsymbol{p})=d_\lambda$ 
for any $\boldsymbol{p}\in\mathbb M$.
\end{itemize}
\end{lemma}

\begin{proof}
A standard base-change argument shows that 
the function $\psi_\lambda$ is independent of basis. 

For any isometry $\zeta\in\mathbb K$, the elements \begin{equation*}
\phi_{\lambda,1}\circ\zeta
,\cdots,\phi_{\lambda,d_\lambda}\circ\zeta\end{equation*} are linearly 
independent, since any linear relation \begin{equation*}0=a_1\phi_{\lambda,1}\circ\zeta
+\cdots+a_{d_\lambda}\phi_{\lambda,d_\lambda}\circ\zeta\,,\end{equation*} 
together with the $\mathbb K$-invariance of $\sigma^{\mathbb M}$, implies 
\begin{align*}0&=\sum_{j_1,j_2=1}^{d_\lambda}a_{j_1}
\overline{a}_{j_2}\int_{\mathbb M}\phi_{\lambda,j_1}(\zeta(\boldsymbol{x}))
\overline{\phi_{\lambda,j_2}(\zeta(\boldsymbol{x}))}~d\sigma^{\mathbb M}
(\boldsymbol{x})\\ &=\sum_{j_1,j_2=1}^{d_\lambda}a_{j_1}
\overline{a}_{j_2}\int_{\mathbb M}\phi_{\lambda,j_1}(\boldsymbol{x})
\overline{\phi_{\lambda,j_2}(\boldsymbol{x})}~d\sigma^{\mathbb M}
(\boldsymbol{x})\\ &=|a_1|^2+\cdots+|a_{d_\lambda}|^2\,,\end{align*} 
which is equivalent to $a_1=\cdots=a_{d_\lambda}=0$. By invariance 
of Casimir-Laplace-Beltrami, the elments \begin{equation*}\phi_{\lambda,1}
\circ\zeta,\cdots,\phi_{\lambda,d_\lambda}\circ\zeta\end{equation*} 
are contained in ${\mathcal H}_\lambda$, and hence, form a 
basis of ${\mathcal H}_\lambda$; furthermore, this must be an orthonormal 
basis by $\mathbb K$-invariance of $\sigma^{\mathbb M}$. By the 
basis-independence character of $\psi_\lambda$, it follows that $\psi_\lambda
\circ(\zeta,\zeta)=\psi_\lambda$. 

For the third part, we notice that orthonomality of $\{\phi_{\lambda,j}
\}_{j=1}^{d_\lambda}$ implies \begin{align*}
d_\lambda&=\sum_{j=1}^{d_\lambda}\int_{\mathbb M}\phi_{\lambda,j}
(\boldsymbol{x})\overline{\phi_{\lambda,j}(\boldsymbol{x})}~
d\sigma^{\mathbb M}(\boldsymbol{x})\\ &
=\int_{\mathbb M}\psi_{\lambda}(\boldsymbol{x}, \boldsymbol{x})~
d\sigma^{\mathbb M}(\boldsymbol{x})\\ &=\psi_\lambda(
\boldsymbol{p},\boldsymbol{p})\,,\end{align*} where the last step 
follows from the identity $\psi_\lambda\circ(\zeta,\zeta)=\psi_\lambda$. 
This proves the lemma.
\end{proof}

In the following, we will closely (but not identically) follow the arguments 
in \cite{MR4138637}. The following lemma proves the main approximation 
result for the $\operatorname{L}^2$-mass of the heat kernel.

\begin{lemma}\label{tail-ineq}
Suppose that $\epsilon\in (0,2^{-e})$; for any $\eta>0$, let 
$\lambda_\infty\geq 4^{\frac{k_\eta}{d_{\mathbb M}}}$ --- where \begin{equation}
\label{bound}k_\eta:=\max\begin{pmatrix}
2+2{\log_2{\frac 1{\eta}}},~ \frac 12d_{\mathbb M}\log_2\begin{pmatrix}
{\frac{d_{\mathbb M}}{2\epsilon^2}}\end{pmatrix}
+d_{\mathbb M}\log_2\log_2\begin{pmatrix} {\frac{d_{\mathbb M}}
{2\epsilon^2}}\end{pmatrix}\end{pmatrix};\end{equation} then the following inequality 
holds for any $\boldsymbol{p}\in\mathbb M$: 
\begin{equation*}|\!|\operatorname{H}_{\boldsymbol{p}}(\boldsymbol{x},
\epsilon^2)-\operatorname{H}^{\lambda_\infty}_{\boldsymbol{p}}(\boldsymbol{x},
\epsilon^2)|\!|_{\operatorname{L}^2}^2\leq 
C_{\mathbb M}\eta^2\,.\end{equation*}
\end{lemma}

\begin{proof}
Write ${\mathscr E}_{\mathbb M,\lambda_\infty}:=
\{\lambda\in {\mathscr E}_{\mathbb M}:\lambda>\lambda_\infty\}$. 
Note that, by \cref{eq:heat2}, integration with respect 
to the Haar-induced probability measure on $\mathbb M$ yields 
\begin{align}\label{eq:heat20}|\!|\operatorname{H}_{\boldsymbol{p}}
(\boldsymbol{x},\epsilon^2)-\operatorname{H}^{\lambda_\infty}_{\boldsymbol{p}}
(\boldsymbol{x},\epsilon^2)|\!|_{\operatorname{L}^2}^2&= \sum_{\lambda\in 
{\mathscr E}_{\mathbb M,\lambda_\infty}
}e^{-2\lambda \epsilon^2}\sum_{j\in [d_\lambda
]}|\phi_{\lambda,j}(\boldsymbol{p})|^2\nonumber\\ 
\mbox{Lemma}~\ref{el1}~\Rightarrow 
\hspace{0.5cm}&\leq 
\sum_{\lambda\in {\mathscr E}_{\mathbb M}(\lambda_\infty)
}e^{-2\lambda \epsilon^2}d_\lambda\,.
\end{align} 
For any $k\in\mathbb Z_+$, write \begin{equation*}{\mathcal I}_k:=
\left(4^{\frac k{d_{\mathbb M}}},4^{\frac{k+1}{d_{\mathbb M}}}\right]\,.\end{equation*} One has 
\begin{align*}\sum_{\lambda\in {\mathscr E}_{\mathbb M,\lambda_\infty}}
e^{-2\lambda \epsilon^2}d_\lambda&\leq \sum_{k\geq k_\eta}\sum_{\lambda\in {\mathcal I}_k}
e^{-2\lambda \epsilon^2}d_{\lambda}\\  &\leq 
\sum_{k\geq k_\eta}\left(\sup_{\lambda\in {\mathcal I}_k}
e^{-2\lambda \epsilon^2}\right)\left(\sum_{\lambda\in {\mathcal I}_k}
d_{\lambda}\right)\\  \cref{Weyls}~\Rightarrow\hspace{1cm}&
\leq C_{\mathbb M}
\sum_{k\geq k_\eta}2^{1+k}e^{-(2^{\frac {2k+d_{\mathbb M}}{d_{\mathbb M}}})\epsilon^2}\,.
\end{align*} Suppose that an integer $k\geq 0$ satisfies 
$2k\geq d_{\mathbb M}\log_2\left({\frac k{\epsilon^2}}\right)$; 
then we get \begin{equation*}e^{-(2^{\frac {2k+
d_{\mathbb M}}{d_{\mathbb M}}}){\epsilon^2}}=(e^{-2^{\frac {2k}
{d_{\mathbb M}}})^{2\epsilon^2}} \leq e^{-2k}
\,.\end{equation*} Consider the inequality \begin{equation}\label{misc101}
{\frac{k}{\log_2\begin{pmatrix}{\frac k{\epsilon^2}}
\end{pmatrix}}}\geq \frac {d_{\mathbb M}}2\,.\end{equation} 
By monotone property of logarithm, the following inequality 
is equivalent to \cref{misc101} above: \begin{equation}\label{final}k\begin{pmatrix}{1-
{\frac{\log_2\log_2\begin{pmatrix}{\frac k{\epsilon^2}}\end{pmatrix}}
{\log_2\begin{pmatrix}{\frac k{\epsilon^2}}\end{pmatrix}}}}
\end{pmatrix}\geq \frac {d_{\mathbb M}}2\log_2\begin{pmatrix} {\frac {d_{\mathbb M}}{2\epsilon^2}}
\end{pmatrix}\,.\end{equation} For $u\in [2^e,+\infty)$, the function 
\[g(u)=1-{\frac{\log_2\log_2u}{\log_2u}}\] satisfies $0<g(u)<1$, has 
global minima $g(2^e)=1-e^{-1}\log_2e>0.46$, and is increasing. 
Since $\epsilon\in (0,2^{-e})$ and $d_{\mathbb M}\geq 1$, the condition 
$d_{\mathbb M}/{2\epsilon^2}>2^e$ is satisfied; requiring $k\geq 
\frac 12d_{\mathbb M}\log_2(2^ed_{\mathbb M})$, we see that 
the following inequality implies \cref{final}: \begin{align}\label{ult}k\geq 
 \frac {d_{\mathbb M}}2\log_2\begin{pmatrix}{\frac {d_{\mathbb M}}{2\epsilon^2}}\end{pmatrix}
\begin{pmatrix}{1-{\frac{\log_2\log_2\begin{pmatrix}{\frac {d_{\mathbb M}}{2\epsilon^2}}
\end{pmatrix}}{\log_2\begin{pmatrix}{\frac {d_{\mathbb M}}{2\epsilon^2}}\end{pmatrix}}}}
\end{pmatrix}^{-1}\,.\end{align} Notice that, for $d_{\mathbb M}\geq 1$ and $\epsilon\in (0,
2^{-e})$, one has $\log_2\left({\frac {d_{\mathbb M}}{2\epsilon^2}}\right)\geq 
2\log_2\log_2\left({\frac {d_{\mathbb M}}{2\epsilon^2}}\right)$, 
which yields the following inequality: \begin{align*}{\frac{\log_2
\left({\frac {d_{\mathbb M}}{2\epsilon^2}}\right)+2\log_2\log_2
\left({\frac {d_{\mathbb M}}{2\epsilon^2}}\right)}{\log_2\left(
{\frac {d_{\mathbb M}}{2\epsilon^2}}\right)}}&=1+{\frac{2\log_2\log_2
\left({\frac {d_{\mathbb M}}{2\epsilon^2}}\right)}{\log_2\left(
{\frac {d_{\mathbb M}}{2\epsilon^2}}\right)}}\\ &\geq\left({1-
{\frac{\log_2\log_2\left({\frac {d_{\mathbb M}}{2\epsilon^2}}\right)}
{\log_2\left({\frac {d_{\mathbb M}}{2\epsilon^2}}\right)}}}\right)^{-1}\\
&= {\frac{\log_2\left({\frac {d_{\mathbb M}}{2\epsilon^2}}\right)} 
{\log_2\left({\frac {d_{\mathbb M}}{2\epsilon^2}}\right)-\log_2\log_2
\left({\frac {d_{\mathbb M}}{2\epsilon^2}}\right)}}\,.\end{align*} 
This shows that \cref{ult} follows if $k_\eta\geq \frac 12d_{\mathbb M}\log_2\begin{pmatrix}
{\frac {d_{\mathbb M}}{2\epsilon^2}}\end{pmatrix}
+{d_{\mathbb M}}\log_2\log_2\begin{pmatrix} {\frac 
{d_{\mathbb M}}{2\epsilon^2}}\end{pmatrix}$, in which case 
we have \begin{align*}|\!|\operatorname{H}_{\boldsymbol{p}}
(\boldsymbol{x},\epsilon^2)-\operatorname{H}^{(\lambda_\infty)}_{\boldsymbol{p}}
(\boldsymbol{x},\epsilon^2)|\!|_{\operatorname{L}^2}^2&=
C_{\mathbb M}\sum_{\lambda\geq 4^{\frac {k_\eta}{d_{\mathbb M}}}}e^{-2\lambda {\epsilon^2}}
d_\lambda\\ &\leq C_{\mathbb M}\sum_{k\geq k_\eta}2^{1+k}
e^{-(2^{\frac {2k+d_{\mathbb M}}{d_{\mathbb M}}}){\epsilon^2}}\\ &\leq C_{\mathbb M}
\sum_{k\geq k_\eta}2^{1+k-2k}\\ 
&\leq 4C_{\mathbb M}2^{-k_\eta}\\ &\leq C_{\mathbb M}\eta^2\,.\end{align*} 
This finishes the proof.
\end{proof}

\begin{remark}\label{0rem01}
Suppose that $d_{\mathbb M}\geq 2^e$, and that $\epsilon\in 
(0,d_{\mathbb M}^{-1})$. It is immediate that 
\begin{equation*}\frac 12d_{\mathbb M}\log_2\begin{pmatrix}
{\frac {d_{\mathbb M}}{2\epsilon^2}}\end{pmatrix}
<\frac 34d_{\mathbb M}\log_2{\frac 1{\epsilon^2}}\,.\end{equation*} Moreover, 
it follows from elementary calculus that \begin{align*}
\frac{2\log_2\log_2\begin{pmatrix} {\frac 
{d_{\mathbb M}}{2\epsilon^2}}\end{pmatrix}}{\log_2
{\frac 1{\epsilon^2}}}&< \frac{2\log_2\log_2\begin{pmatrix} {\frac 
1{2\epsilon^3}}\end{pmatrix}}{\log_2
{\frac 1{\epsilon^2}}}\hspace{1cm}\because~\epsilon
<d_{\mathbb M}^{-1}\\ &=\frac{\log_2\log_2\begin{pmatrix} {\frac 
1{2\epsilon^3}}\end{pmatrix}}{\log_2
{\frac 1{\epsilon}}}\\ &<\frac{3\log_2\log_2\begin{pmatrix} {\frac 
1{\epsilon}}\end{pmatrix}}{\log_2
{\frac 1{\epsilon}}}\\ &<\frac{3\log_2\log_2\begin{pmatrix}
d_{\mathbb M}\end{pmatrix}}{\log_2\begin{pmatrix}
d_{\mathbb M}\end{pmatrix}}\,.\end{align*} 
This shows \begin{equation*}\frac 12d_{\mathbb M}\log_2\begin{pmatrix}
{\frac {d_{\mathbb M}}{2\epsilon^2}}\end{pmatrix}
+{d_{\mathbb M}}\log_2\log_2\begin{pmatrix} {\frac 
{d_{\mathbb M}}{2\epsilon^2}}\end{pmatrix}<\frac 34d_{\mathbb M}\log_2
{\frac 1{\epsilon^2}}\begin{pmatrix}1+\frac{2\log_2\log_2\begin{pmatrix}
d_{\mathbb M}\end{pmatrix}} {\log_2\begin{pmatrix} 
d_{\mathbb M}\end{pmatrix}}\end{pmatrix}\,.\end{equation*}
\end{remark}

The following lemma furnishes an analytically checkable criterion for a subset 
${\mathcal S}\subseteq\mathbb M$ --- in any compact 
Riemannian manifold $\mathbb M$ --- to be an $\epsilon$-cover. The underlying 
idea is that, if the weighted sum of the heat kernels based at points 
in ${\mathcal S}$ is close to the identity in $\operatorname{L}^2$-norm, then 
the set $\mathcal S$ must be a cover.

\begin{lemma}\label{lem:31-10-sc1}
Let ${\mathcal S}\subseteq \mathbb M$ be a nonempty subset. 
For $\epsilon\in(0,1)$, if \begin{displaymath}r_{\epsilon,\mathbb M}:=r(\epsilon,1)=2\epsilon
\sqrt{\ln {\frac {3C_{\mathbb M}}{\epsilon^{2d_{\mathbb M}
-\bar{d}_{\mathbb M}-1}}}}\end{displaymath} is 
sufficiently small, then the following inequality 
implies that ${\mathcal S}
\subseteq \mathbb M$ is a $2r_{\epsilon,\mathbb M}$-net: \begin{equation}\label{neteq}
\Bigg\lvert\!\Bigg\lvert1_{\mathbb M}(\boldsymbol{x})-
{\frac 1{|{\mathcal S}|
}}\sum_{\boldsymbol{p}\in{\mathcal S}}\operatorname{H}_{\boldsymbol{p}}
(\boldsymbol{x},{\epsilon^2})\Bigg\rvert\!\Bigg
\rvert_{\operatorname{L}^2}\leq {\frac{r_{\epsilon,\mathbb M}^{\frac {d_{\mathbb M}}2}}3}
\sqrt{\frac {{\mathfrak v}_{\mathbb M}\pi^{\frac {d_{\mathbb M}}2}}{\Gamma\left(
{\frac {d_{\mathbb M}}2}+1\right)}}\end{equation}
\end{lemma}

\begin{remark}
\cref{lem:31-10-sc1} first appeared in \cite{MR4138637}, where 
the statement was formulated in terms of compact Lie groups.
\end{remark}

\begin{proof}
From \cref{heat-ineq-rem1}, it follows that the inequalities 
\begin{equation}\label{2-side}\frac23\cdot1_{\mathbb M}(\boldsymbol{x})
\leq 1_{\mathbb M}(\boldsymbol{x})-
\operatorname{H}_{\boldsymbol{p}}(\boldsymbol{x}, 
\epsilon^2) \leq 1_{\mathbb M}(\boldsymbol{x})\end{equation} 
hold for any $\boldsymbol{x},\boldsymbol{p}\in
\mathbb M$ satisfying $\partial_{\mathbb M}(\boldsymbol{x},
\boldsymbol{p})\geq r_{\epsilon,\mathbb M}$. 
Write $\operatorname{B}(\boldsymbol{p},r_{\epsilon,\mathbb M})\subseteq 
\mathbb M$ for the open geodesic disk of radius $r_{\epsilon,\mathbb M}$, 
centered at ${\boldsymbol{p}}\in \mathbb M$; 
thus, $\operatorname{B}(\boldsymbol{x},r_{\epsilon,\mathbb M}):=
\{\boldsymbol{x}\in\mathbb M: \partial_{\mathbb M}(\boldsymbol{x},
\boldsymbol{p})<r_{\epsilon,\mathbb M}\}$. 
Also, let $\operatorname{B}_{d_{\mathbb M}}
\subseteq \mathbb R^{d_{\mathbb M}}$ denote 
the origin-centric unit euclidean ball; one has (see \cite{MR2024928})
\begin{equation}\label{volumes}\lim_{r\rightarrow 0}
{\frac{\sigma^{\mathbb M}(\operatorname{B}(\boldsymbol{x},
r_{\epsilon,\mathbb M}))}{r_{\epsilon,\mathbb M}^{d_{\mathbb M}}}}
={\mathfrak v}_{\mathbb M}\cdot \operatorname{vol}(\operatorname{B}_{d_{\mathbb M}})
=\frac {{\mathfrak v}_{\mathbb M}\pi^{\frac {d_{\mathbb M}}2}}{\Gamma\left(
{\frac {d_{\mathbb M}}2}+1\right)}\end{equation} 
Now assume, if possible, that 
${\mathcal S}\subseteq \mathbb M$ is not a $2r_{\epsilon,\mathbb M}$-net, so that there 
is ${\boldsymbol{p}_0}\in \mathbb M$ with $d(\boldsymbol{p}_0
,{\mathcal S})>2r_{\epsilon,\mathbb M}$. By triangle inequality applied to the 
Riemannian distance on $\mathbb M$, it follows that $\operatorname{B}(\boldsymbol{p}_0,r_{\epsilon,\mathbb M})
\cap \operatorname{B}(\boldsymbol{p},r_{\epsilon,\mathbb M})=\emptyset$ for 
every $\boldsymbol{p}\in{\mathcal S}$. Writing \begin{equation*}
\alpha_{d_{\mathbb M}}:=\sqrt{\frac {{\mathfrak v}_{\mathbb M}
\pi^{\frac {d_{\mathbb M}}2}}{\Gamma\left({\frac {d_{\mathbb M}}2}
+1\right)}}\,,\end{equation*} we derive from \cref{neteq} 
and \cref{2-side} the following: \begin{align*}{\frac{
r_{\epsilon,\mathbb M}^{\frac {d_{\mathbb M}}2}\alpha_{d_{\mathbb M}}}3}&\geq \Big\lvert\!\Big\lvert
1_{\mathbb M}(\boldsymbol{x})-{\frac 1{|{\mathcal S}|}}
\sum_{\boldsymbol{p}\in{\mathcal S}}\operatorname{H}_{
\boldsymbol{p}}(\boldsymbol{x},\epsilon^2)\Big\rvert\!\Big
\rvert_{\operatorname{L}^2}\\ &= \Bigg\lvert\!\Bigg\lvert{\frac 1
{|{\mathcal S}|}}\sum_{\boldsymbol{p}
\in{\mathcal S}}\left\vert{1_{\mathbb M}(\boldsymbol{x})
-\operatorname{H}_{\boldsymbol{p}}(\boldsymbol{x},
\epsilon^2)}\right\vert\Bigg\rvert\!\Bigg\rvert_{\operatorname{L}^2}
\\ &\geq {\frac 23}\left({\int_{\operatorname{B}(\boldsymbol{p}_0,r_{\epsilon,\mathbb M})}
d\sigma^{\mathbb M}(\boldsymbol{x})}\right)^{\frac 12}\,.
\end{align*} This implies \begin{align*}{\frac{\sigma^{\mathbb M}
(\operatorname{B}(\boldsymbol{p}_0,r_{\epsilon,\mathbb M}))}{r_{\epsilon,\mathbb M}^{d_{\mathbb M}
}}}&={\frac 1{r_{\epsilon,\mathbb M}^{d_{\mathbb M}}}}\int_{\operatorname{B}
(\boldsymbol{p}_0,r_{\epsilon,\mathbb M})}d\sigma^{\mathbb M}(\boldsymbol{x})\\ 
&\leq {\frac {\alpha_{d_{\mathbb M}}^2}4}\end{align*} Considering 
\cref{volumes}, this is impossible if $r_{\epsilon,\mathbb M}>0$ is sufficiently small.
\end{proof}

\bigskip

\section{Equidistributed cover}
In what follows, we will make use of the following tail-bound on operator-valued
random variables on a finite dimensional Hilbert space.
\begin{lemma}[Ahlswede-Winter]\label{tail}
Let $V$ be a finite dimensional (real or complex) 
Hilbert space, with $\dim V=D$. Let $A_1,
\cdots,A_k$ be independent identically distributed random variables taking 
values in the cone of positive semidefinite operators on $V$, 
such that $A_j\preceq \operatorname{I}$ for each $j\in [k]$, and there is 
some real $\mu\geq 0$ for which $\mathbb E[A_j]=A\succeq \mu \operatorname{I}$ for 
each $j\in [k]$. Then, for all $\upsilon\in [0,0.5]$, the following holds: 
\begin{equation}\label{taileq}\mathbb P\left({{\frac 1k}
\sum_{j=1}^kA_j\notin [(1-\upsilon)A,(1+\upsilon)A]}\right)
\leq 2D\exp\begin{pmatrix}{\frac{-\upsilon^2\mu k}{2\ln 2}}\end{pmatrix}\end{equation}\qed
\end{lemma}

Now let ${\mathcal S}\subset \mathbb K$ be a non-empty subset. Recall 
that, for any $\lambda_\infty>0$, we write ${\mathcal E}^\star_{\lambda_\infty}
(\mathbb M)\subseteq \operatorname{L}^2_0(\mathbb M)$ for the direct 
sum \begin{equation*}{\mathcal E}^\star_{\lambda_r}(\mathbb M)
=\bigoplus_{\lambda\in {\mathscr E}_{\mathbb M}(0,\lambda_\infty)}
{\mathcal H}_{\lambda}\,,\end{equation*} where ${\mathscr E}_{\mathbb M}
(0,\lambda_\infty)=\{\lambda\in {\mathscr E}_{\mathbb M}:0<\lambda\leq\lambda_\infty\}$. 
Note that \begin{equation}\label{dime}\dim {\mathcal E}^\star_{\lambda_\infty}(\mathbb M)< \dim {\mathcal E}_{\lambda_\infty}(\mathbb M)\leq C_{\mathbb M}
\lambda_\infty^{\frac {d_{\mathbb M}}2}\end{equation} by \cref{Weyls}.
Because $\Delta_{\mathbb M}$ is $\mathbb K$-invariant,  
the subspace ${\mathcal E}_{\lambda_\infty}^\star(\mathbb M)$ is invariant under the 
operators $A_{\boldsymbol{s}}$ for all $\boldsymbol{s}\in \mathbb K$, 
where $A_{\boldsymbol{s}}$ is defined via \begin{equation}\label{ops}
A_{\boldsymbol{s}}(\phi)(\boldsymbol{x}):={\frac 12}\phi(\boldsymbol{x})
+{\frac 14}(\phi(\boldsymbol{s}\boldsymbol{x})+\phi(\boldsymbol{s}^{-1}
\boldsymbol{x}))\,.\end{equation} Due to $\mathbb K$-invariance of the 
measure $\sigma^{\mathbb M}$, the operators $A_{\boldsymbol{s}}:
{\mathcal E}_{\lambda_\infty}(\mathbb M)\rightarrow {\mathcal E}_{\lambda_\infty}
(\mathbb M)$ turns out to be self-adjoint. Positive semidefiniteness of 
$A_{\boldsymbol{s}}$ follows from the identity \begin{equation*}
\langle A_{\boldsymbol{s}}\phi, \phi\rangle=
{\frac 14}\int_{\mathbb M}\left|{\phi(\boldsymbol{x})+\phi(\boldsymbol{s}
\boldsymbol{x})}\right|^2d\sigma^{\mathbb M}(\boldsymbol{x})\,.
\end{equation*} Moreover, writing $\sigma^{\mathbb K}$ for the 
unique invariant Haar (probability) measure on $\mathbb K$, 
one has \begin{align*}\left({{\mathbb E}_{\boldsymbol{s}\sim\sigma^{\mathbb M}}
[A_{\boldsymbol{s}}]}\right)\phi(\boldsymbol{x})={\frac 12}\phi(\boldsymbol{x})
+{\frac 14}\int_{\mathbb K}\phi(\boldsymbol{s}\boldsymbol{x})
~d\sigma^{\mathbb K}(\boldsymbol{s})+{\frac 14}\int_{\mathbb K}
\phi(\boldsymbol{s}^{-1}\boldsymbol{x})~d\sigma^{\mathbb K}
(\boldsymbol{s})\,.\end{align*} Writing $\tau: \mathbb K\rightarrow \mathbb M$ for the map 
$\boldsymbol{s}\mapsto \boldsymbol{s}\boldsymbol{o}$, 
one has \begin{align*}\int_{\mathbb K}\phi(\boldsymbol{s}\boldsymbol{o})~d
\sigma^{\mathbb K}(\boldsymbol{s})&=\int_{\mathbb K}(\phi\circ\tau)
(\boldsymbol{s})~d\sigma^{\mathbb K}(\boldsymbol{s})\\ &=\int_{\mathbb M}
\phi(\boldsymbol{x})~d(\tau_\ast \sigma^{\mathbb K})\,,\end{align*} 
where $\tau_\ast \sigma^{\mathbb K}$ is the push-forward measure, 
given by $\tau_\ast \sigma^{\mathbb K}(E)=\sigma^{\mathbb K}(\tau^{-1}(E))$ 
for Borel subsets $E\subseteq\mathbb M$. Since $\tau_\ast\sigma^{\mathbb K}$ 
is $\mathbb K$-invariant Borel probability measure on $\mathbb M$, 
one has $\sigma^{\mathbb M}=\tau_\ast\sigma^{\mathbb K}$. 
Because $\phi\in {\mathcal E}_{\lambda_\infty}(\mathbb M)\subset
\operatorname{L}^2(\mathbb M)$, one has \begin{equation*}\int_{\mathbb K}
\phi(\boldsymbol{x}\boldsymbol{s})~d\sigma^{\mathbb K}(\boldsymbol{s})=
\int_{\mathbb M}\phi(\boldsymbol{x})~d\sigma^{\mathbb M}=0\,.\end{equation*} 
Therefore, $\left({\mathbb E}_{\boldsymbol{s}\sim\sigma^{\mathbb K}}
[A_{\boldsymbol{s}}]\right)\phi(\boldsymbol{x})={\frac 12}\phi(\boldsymbol{x})$, 
which makes \begin{align*}\label{expect}{\mathbb E}_{\boldsymbol{s}\sim
\sigma^{\mathbb K}}[A_{\boldsymbol{s}}]={\frac 12}\operatorname{I},\hspace{1cm}
\boldsymbol{s}\in \mathbb K\end{align*} Furthermore, the operators $\operatorname{I}
-A_{\boldsymbol{s}}$ are positive semidefinite for all $\boldsymbol{s}\in 
\mathbb K$, because \begin{align*}\langle \phi-A_{\boldsymbol{s}}\phi,
\phi\rangle ={\frac 14}\int_{\mathbb M}\left|{\phi(\boldsymbol{x})-\phi(
\boldsymbol{x}\boldsymbol{s})}\right|^2d\sigma^{\mathbb M}(\boldsymbol{x})
\,.\end{align*}

\medskip

\subsection{Random cover}\label{ssec:1}
In the following lemma, we derive a high probability upper-bound on the 
$\operatorname{L}^2$-norm of the non-constant tail of the heat kernel 
on $\mathbb M$.
 
\begin{lemma}\label{eqdist}
Let ${\mathcal S}\subset \mathbb K$ be a nonempty finite multisubset whose 
elements are selected independently at random from the Haar measure 
on $\mathbb K$ and let ${\hat {\mathcal S}}:={\mathcal S}\sqcup 
{\mathcal S}^{-1}$ be the (multi)set of all elements in ${\mathcal S}$ 
and their inverses. Suppose that $\eta>0$ satisfies \begin{equation}\label{misc1}
2+2\log_2{\frac 1{\eta}}\geq \frac 12d_{\mathbb M}\log_2\begin{pmatrix}
{\frac{d_{\mathbb M}}{2\epsilon^2}}\end{pmatrix}
+d_{\mathbb M}\log_2\log_2\begin{pmatrix}{\frac{d_{\mathbb M}}
{2\epsilon^2}}\end{pmatrix}
\,.\end{equation} 
Let $\epsilon\in (0,2^{-e})$; if \begin{equation*}\delta:=
{\frac {C_{\mathbb M}}{\eta}}\exp\begin{pmatrix}{\frac{-|{\mathcal S}|}
{16\ln 2}}\end{pmatrix}\,,\end{equation*} then 
for any integer $\ell>0$, and any $\boldsymbol{p}\in\mathbb M$, 
the following inequality holds: \begin{equation}
\label{eqdisteq}\operatornamewithlimits{\mathbb P}_{\mathcal S} \left(\Bigg|\!\Bigg|
{\frac 1{|\hat{\mathcal S}|^\ell}}\sum_{\boldsymbol{s}\in {\hat {\mathcal S}}^\ell} 
\operatorname{H}_{\boldsymbol{p}}(\boldsymbol{s}\boldsymbol{x},\epsilon^2)-
1_{\mathbb M}(\boldsymbol{x})\Bigg|\!\Bigg|_{\operatorname{L}^2}
\leq C_{\mathbb M}\left(2^{-\ell}\epsilon^{-d_{\mathbb M}+\frac{
\bar{d}_{\mathbb M}+1}2}+\eta\right)\right)\geq 1-2\delta\,.\end{equation}
\end{lemma}

\begin{proof}
Setting $\upsilon=\mu=0.5$ and $\lambda_\infty=\eta^{-\frac 2d}$ in \cref{taileq} 
yields \begin{align*}\operatornamewithlimits{\mathbb P}_{\mathcal S}
\left({{\frac 1{|{\mathcal S}|}}\sum_{\boldsymbol{s}\in {\mathcal S}}A_{\boldsymbol{s}}
\notin \left[{\frac 14}\operatorname{I}, {\frac 34}\operatorname{I}\right]}\right)\leq 
{\frac {2C_{\mathbb M}}{\eta}}\exp\left({\frac{-|{{\mathcal S}}|}{16\ln 2}}\right)\,.\end{align*} 
Note that the combination of \cref{misc1} and \cref{dime} --- together with the choice of $\lambda_\infty=\eta^{-\frac 2d}$ ---  ensures that \cref{tail-ineq} can be applied if required. 
Now, disentangling the inequality $\frac 14\operatorname{I} \preccurlyeq |{\mathcal S}|^{-1}\sum_{\boldsymbol{s}\in{\mathcal S}} A_{\boldsymbol{s}}\preccurlyeq \frac 34\operatorname{I}$ in the cone of  positive-definite operators on ${\mathcal E}^\star_{\lambda_\infty}(\mathbb M)$, we obtain \begin{equation*}-\frac 14\operatorname{I} \preccurlyeq \left({\frac 1{|{\mathcal S}|}}\sum_{\boldsymbol{s}\in{\mathcal S}} A_{\boldsymbol{s}}-\frac 12\operatorname{I}\right)\preccurlyeq \frac 14\operatorname{I}\,.\end{equation*} Thus, the eigenvalues of the hermitian operator $|{\mathcal S}|^{-1}\sum_{\boldsymbol{s}\in{\mathcal S}} A_{\boldsymbol{s}}-\frac 12\operatorname{I}$ are all in $[-\frac 14,\frac 14]$, which shows that --- for all $\phi\in {\mathcal E}^\star_{\lambda_\infty}(\mathbb M)$ ---  the following holds: \begin{equation*}\Bigg|\!\Bigg|\frac 1{|{\mathcal S}|}\sum_{\boldsymbol{s}\in {\mathcal S}} \left(\frac 14\phi_{\boldsymbol{s}}+\frac 14\phi_{\boldsymbol{s}^{-1}}\right)\Bigg|\!\Bigg|_{\operatorname{L}^2}^2\leq \frac 14|\!|\phi|\!|_{\operatorname{L}^2}^2\,.\end{equation*} In particular, one has the concentration inequality \begin{align*}\operatornamewithlimits{\mathbb 
P}_{\mathcal S}\left({\Bigg|\!\Bigg|{\frac 1{|{{\mathcal S}}|}}\sum_{\boldsymbol{s}\in 
\hat{\mathcal S}}{\frac 14}\phi_{\boldsymbol{s}}\Bigg|\!\Bigg|_{\operatorname{L}^2}
\leq {\frac 14}|\!|\phi|\!|_{\operatorname{L}^2}~\forall~\phi\in 
{\mathcal E}^\star_{\lambda_\infty}(\mathbb M)}\right)\geq 1-{\frac {2C_{\mathbb M}}
{\eta}}\exp\left({\frac{-{|{\mathcal S}|}}{16\ln 2}}\right)\,,\end{align*} where we denote 
$\phi_{\boldsymbol{s}}(\boldsymbol{x})=\phi(\boldsymbol{s}\boldsymbol{x})$ for all 
$\boldsymbol{x}\in\mathbb M$. Now, writing ${\operatorname{H
}}^{\star,\lambda_\infty}_{\boldsymbol{p}}$ for the function $\boldsymbol{x}\mapsto 
\operatorname{H}_{\boldsymbol{p}}^{\lambda_\infty}(\boldsymbol{x},\epsilon^2)
-1_{\mathbb M}(\boldsymbol{x})$, one has \begin{align*}
\operatornamewithlimits{\mathbb P}_{\mathcal S}\left({\Bigg|\!\Bigg|
{\frac 1{|{\hat {\mathcal S}}|}}\sum_{\boldsymbol{s}\in {\hat {\mathcal S}}}
{\operatorname{H}}^{\star,\lambda_\infty}_{\boldsymbol{s}\boldsymbol{p}}\Bigg|\!
\Bigg|_{\operatorname{L}^2}\leq {\frac 12}|\!|{\operatorname{H
}}^{\star,\lambda_\infty}_{\boldsymbol{p}}|\!|_{\operatorname{L}^2}}\right)\geq 1
-{\frac {2C_{\mathbb M}}{\eta}}\exp\left({\frac{-|{{\mathcal S}}|}{16\ln 2}}\right)\,.\end{align*} 
Iterating this inequality $\ell>0$ times fetches \begin{equation}\label{iterate}
\operatornamewithlimits{\mathbb P}_{\mathcal S}\left({\Bigg|\!\Bigg|{\frac 1{|{\hat {\mathcal S}}|^\ell}}
\sum_{\boldsymbol{s}\in {\hat {\mathcal S}}^\ell} {\operatorname{H
}}^{\star,\lambda_\infty}_{\boldsymbol{s}\boldsymbol{p}}\Bigg|\!\Bigg|_{\operatorname{L}^2}\leq 
{\frac 1{2^\ell}}\Bigg|\!\Bigg|{\operatorname{H}}^{\star,\lambda_\infty}_{\boldsymbol{p}}
\Bigg|\!\Bigg|_{\operatorname{L}^2}}\right)\geq 1-{\frac {2C_{\mathbb M}}
{\eta}}\exp\left({\frac{-|{{\mathcal S}}|}{16\ln 2}}\right)\,.\end{equation} Let 
${\operatorname{H}}^{\star}$ denote the function $\boldsymbol{x}\mapsto 
\operatorname{H}_{\boldsymbol{p}}(\boldsymbol{x},\epsilon^2)-
1_{\mathbb M}(\boldsymbol{x})$; then we have \begin{align*}{\operatorname{H}}^{\star}
&=({\operatorname{H}}^{\star}-{\operatorname{H}}^{\star,\lambda_\infty})
+{\operatorname{H}}^{\star,\lambda_\infty}\\ &=({\operatorname{H}}
-{\operatorname{H}}^{\lambda_\infty})+{\operatorname{H}}^{\star,
\lambda_\infty}\,,\end{align*} where ${\operatorname{H}}^{\lambda_\infty}
=1_{\mathbb M}(\boldsymbol{x})+{\operatorname{H}}^{\star,\lambda_\infty}_{\boldsymbol{p}}$ and 
${\operatorname{H}}=1_{\mathbb M}(\boldsymbol{x})+{\operatorname{H}}^{\star}$. 
Then, using triangle inequality --- alongwith \cref{tail-ineq} and 
\cref{trunc} --- in inequality \cref{iterate}, one derives \begin{align*}&
\operatornamewithlimits{\mathbb P}_{\mathcal S} \left(\Bigg|\!\Bigg|
{\frac 1{|\hat{\mathcal S}|^\ell}}\sum_{\boldsymbol{s}\in {\hat {\mathcal S}}^\ell} 
{\operatorname{H}}^{\star}_{\boldsymbol{s}\boldsymbol{p}}\Bigg
|\!\Bigg|_{\operatorname{L}^2}\leq C_{\mathbb M}\left(\eta+2^{-\ell}
\epsilon^{-d_{\mathbb M}+\frac{\bar{d}_{\mathbb M}+1}2}\right)\right)\\ 
\geq~&\operatornamewithlimits{\mathbb P}_{\mathcal S}\left({\Bigg|\!\Bigg|{\frac 1
{|{\hat {\mathcal S}}|^\ell}}\sum_{\boldsymbol{s}\in {\hat {\mathcal S}}^\ell} {\operatorname{H
}}^{\star}_{\boldsymbol{s}\boldsymbol{p}}\Bigg|\!\Bigg|_{\operatorname{L}^2}\leq 2^{-\ell}
|\!|{\operatorname{H}}^{\star,\lambda_\infty}|\!|_{\operatorname{L}^2}}
+|\!|{\operatorname{H}}^{\lambda_\infty}-{\operatorname{H}}|
\!|_{\operatorname{L}^2}\right)\\ \geq ~& \operatornamewithlimits{\mathbb P
}_{\mathcal S}\left({\Bigg|\!\Bigg|{\frac 1{|{\hat {\mathcal S}}|^\ell}}\sum_{\boldsymbol{s}
\in {\hat {\mathcal S}}^\ell} {\operatorname{H}}^{\star,\lambda_\infty}_{
\boldsymbol{s}\boldsymbol{p}}\Bigg|\!\Bigg|_{\operatorname{L}^2}\leq 2^{-\ell}
|\!|{\operatorname{H}}^{\star,\lambda_\infty}|\!|_{\operatorname{L}^2}}
\right)\\ \geq~& 1-{\frac {2C_{\mathbb M}}{\eta}}\exp\left({\frac{-|{{\mathcal S}
}|}{16\ln 2}}\right)\,,\end{align*} which proves the lemma.
\end{proof}

\begin{remark}\label{des-remark}
We note, for use in the next section, that the condition \cref{misc1}
is not used in the proof, except at the very end. In particular, \cref{iterate} 
has been derived independently, without using \cref{misc1}.
\end{remark}

We now prove that, for any $\boldsymbol{p}\in\mathbb M$, and a random 
subset $\mathcal S\subseteq\mathbb M$, the orbit $\hat{\mathcal S}^\ell
\boldsymbol{p}\subseteq \mathbb M$ is an $r_{\epsilon,\mathbb M}$-cover, provided that 
the cardinality of the random subset and the integer $\ell$ are appropriately 
large.

\begin{theorem}\label{mainsym1}
Let ${\mathcal S}\subset \mathbb K$ be a nonempty finite multisubset whose 
elements are selected independently at random from the Haar measure 
on $\mathbb K$. Let $\delta\in (0,\frac 12)$, and assume that $
\epsilon\in (0,2^{-e})$ is sufficiently small. Suppose that 
the cardinality of ${\mathcal S}$ satisfies 
\begin{equation}\label{01seq1}|{\mathcal S}|= 16\ln 2\left(\ln C_{\mathbb M} 
+\frac {d_{\mathbb M}}4\log_2\begin{pmatrix}
{\frac{d_{\mathbb M}}{2\epsilon^2}}\end{pmatrix}
+\frac{d_{\mathbb M}}2\log_2\log_2\begin{pmatrix}{\frac{d_{\mathbb M}}
{2\epsilon^2}}\end{pmatrix}+\ln \frac 1{e\delta}\right)\,.\end{equation} 
Let $\ell>0$ be an integer satisfying \begin{equation}\label{01seq01}
\ell\geq \begin{pmatrix}d_{\mathbb M}-\frac{\bar{d}_{\mathbb M}+1}2\end{pmatrix}
\log_2\frac 1\epsilon+\log_2\begin{pmatrix}\frac{6C_{\mathbb M}}
{{\mathfrak v}_{\mathbb M}}\end{pmatrix}+\frac {d_{\mathbb M}}4\log_2
\begin{pmatrix}\frac 1{\pi r^2_{\epsilon,\mathbb M}}\end{pmatrix}
+\frac 12\log_2\Gamma\begin{pmatrix}{\frac {d_{\mathbb M}}2}+1
\end{pmatrix}\,.\end{equation} 
Then, for any $\boldsymbol{p}_0\in\mathbb M$, the random multisubset 
$\mathscr S:=\hat{\mathcal S}^\ell\boldsymbol{p}_0\subseteq \mathbb M$ is an 
$r_{\epsilon,\mathbb M}$-cover of $\mathbb M$ with with probability at least $1-2\delta$; 
here, \begin{displaymath}r_{\epsilon,\mathbb M}=2\epsilon\sqrt{\ln {\frac {3C_{\mathbb M}}
{\epsilon^{2d_{\mathbb M}-\bar{d}_{\mathbb M}-1}}}}\,.\end{displaymath} 
\end{theorem}

\begin{proof}
Fix $\boldsymbol{p}_0\in\mathbb M$; in order to prove the theorem, 
it suffices --- by \cref{lem:31-10-sc1} --- to 
show that \begin{equation*}\operatornamewithlimits{\mathbb P
}_{\mathcal S}\left(\Bigg\lvert\!\Bigg\lvert1_{\mathbb M}(\boldsymbol{x})-
{\frac 1{|\hat{\mathcal S}|^\ell
}}\sum_{\boldsymbol{p}\in\hat{\mathcal S}^\ell\boldsymbol{p}_0}\operatorname{
H}_{\boldsymbol{p}}(\boldsymbol{x},{\epsilon^2})\Bigg\rvert\!\Bigg
\rvert_{\operatorname{L}^2}\leq {\frac{\alpha_{d_{\mathbb M}}
r_{\epsilon,\mathbb M}^{\frac {d_{\mathbb M}}2}}3}\right)\geq 1-2\delta\,,\end{equation*} 
where $\alpha_{d_{\mathbb M}}$ satisfies $\alpha_{d_{\mathbb M}}^2
\Gamma\left({\frac {d_{\mathbb M}}2}+1\right)={\mathfrak v}_{\mathbb M}
\pi^{\frac {d_{\mathbb M}}2}$.

Suppose that the real number $\eta>0$ and the integer $\ell>0$ 
satisfy the following condition: 
\begin{equation}\label{1eq2}3C_{\mathbb M}(\eta+2^{-\ell}\epsilon^{-d_{\mathbb M}
+\frac {\bar{d}_{\mathbb M}+1}2})\leq \alpha_{d_{\mathbb M}}
r_{\epsilon,\mathbb M}^{\frac {d_{\mathbb M}}2}\,.\end{equation}
Applying the condition in \cref{1eq2} to \cref{eqdist}, we derive 
\begin{align}\label{1eq0}&\operatornamewithlimits{\mathbb P}_{
\mathcal S} \left(\Bigg|\!\Bigg|{\frac 1{|\hat{\mathcal S}|^\ell}}\sum_{
\boldsymbol{s}\in {\hat {\mathcal S}}^\ell} \operatorname{H}_{\boldsymbol{
p}_0}(\boldsymbol{s}\boldsymbol{x},\epsilon^2)-1_{\mathbb M}(\boldsymbol{
x})\Bigg|\!\Bigg|_{\operatorname{L}^2}\leq \frac {\alpha_{d_{\mathbb M}}
r_{\epsilon,\mathbb M}^{\frac {d_{\mathbb M}}2}}3\right)\nonumber\\ \geq~& 1-
{\frac {2C_{\mathbb M}}{\eta}}\exp\left({\frac{-|{{\mathcal S}
}|}{16\ln 2}}\right)\,,\end{align} provided that 
\begin{equation}\label{01misc1}
2+2\log_2{\frac 1{\eta}}\geq \frac 12d_{\mathbb M}\log_2\begin{pmatrix}
{\frac{d_{\mathbb M}}{2\epsilon^2}}\end{pmatrix}
+d_{\mathbb M}\log_2\log_2\begin{pmatrix}{\frac{d_{\mathbb M}}
{2\epsilon^2}}\end{pmatrix}
\,.\end{equation} 
We force an equality in \cref{01misc1}, and demand that 
the inequality $2^{-\ell}\epsilon^{-d_{\mathbb M}+
\frac {\bar{d}_{\mathbb M}+1}2}\leq\eta$ be true, 
so that it suffices to require ---  instead of \cref{1eq2} --- the condition 
\begin{equation}\label{1eq3}2^{-\ell}\epsilon^{-d_{\mathbb M}+\frac {\bar{d
}_{\mathbb M}+1}2}\leq\eta\leq\frac{\alpha_{d_{\mathbb M}}r_{\epsilon,
\mathbb M}^{\frac {d_{\mathbb M}}2}}{6C_{\mathbb M}}\,.\end{equation} 
This is guaranteed if we set \begin{equation}\label{0leq0}\ell\geq \begin{pmatrix}
d_{\mathbb M}-\frac{\bar{d}_{\mathbb M}+1}2\end{pmatrix}\log_2\frac 1\epsilon+\log_2
\begin{pmatrix}\frac{6C_{\mathbb M}}{{\mathfrak v}_{\mathbb M}}\end{pmatrix}
+\frac {d_{\mathbb M}}4\log_2\begin{pmatrix}\frac 1{\pi r^2_{\epsilon,\mathbb M}}\end{pmatrix}
+\frac 12\log_2\Gamma\begin{pmatrix}{\frac {d_{\mathbb M}}2}+1
\end{pmatrix}\,.\end{equation} In fact, since $\epsilon\in (0,2^{-e})$ is sufficiently small, 
we have $\epsilon\ll r_{\epsilon,\mathbb M}$, and thus, 
an equality in \cref{01misc1} ensures that \begin{equation*}
\eta\leq \frac{\alpha_{d_{\mathbb M}}r_{\epsilon,
\mathbb M}^{\frac {d_{\mathbb M}}2}}{6C_{\mathbb M}}\,.\end{equation*} Furthermore, 
the other inequality, $2^{-\ell}\epsilon^{-d_{\mathbb M}+
\frac {\bar{d}_{\mathbb M}+1}2}\leq\eta$, holds because the reverse inequality, 
$2^{-\ell}\epsilon^{-d_{\mathbb M}+\frac {\bar{d}_{\mathbb M}+
1}2}>\eta$, contradicts \cref{0leq0}. 
Finally, set \begin{equation}\label{1eq4}\delta=
{\frac {C_{\mathbb M}}{\eta}}\exp\begin{pmatrix}
{\frac{-|{\mathcal S}|}{16\ln 2}}\end{pmatrix}\,,\end{equation} and recall 
that \cref{01misc1} is an equality, which translates to the condition 
\begin{equation}\label{0seq1}|{\mathcal S}|= 16\ln 2\left(\ln C_{\mathbb M} 
+\frac {d_{\mathbb M}}4\log_2\begin{pmatrix}
{\frac{d_{\mathbb M}}{2\epsilon^2}}\end{pmatrix}
+\frac{d_{\mathbb M}}2\log_2\log_2\begin{pmatrix}{\frac{d_{\mathbb M}}
{2\epsilon^2}}\end{pmatrix}+\ln \frac 1{e\delta}\right)\,.\end{equation} This 
completes the proof.
\end{proof}

\medskip

\subsection{Approximate $(\lambda,2)$-design}\label{ssec:3}

\begin{definition}
For any eigenvalue $\lambda$ of the Laplace-Beltrami $-\Delta_{\mathbb M}$ on ${\mathbb M}$, a finite subset 
${\mathscr S}\subset\mathbb M$ is a $\lambda$-design if for every $\phi\in\mathcal E_{\lambda}(\mathbb M)$, 
the following equality holds: \begin{align}\label{definition:1}\int_{\mathbb M}\phi(\boldsymbol{x})~
\sigma^{\mathbb M}(d\boldsymbol{x})&=\frac 1{|{\mathscr S}|}\sum_{\boldsymbol{x}
\in {\mathscr S}}\phi(\boldsymbol{x})\end{align}
\end{definition}

Recently, \cite{MR4308661} established the existence of $\lambda$-design of ``small'' size on compact Riemannian manifolds. However, their result is not constructive; more generally, efficient algorithmic construction of such designs of optimal size are still far fetched. 

\begin{definition}
Let $\lambda$ be an eigenvalue of the Laplace-Beltrami $-\Delta_{\mathbb M}$. An $\upsilon$-approximate $(\lambda,2)$-design on $\mathbb M$ is a finite subset $\mathscr S\subseteq \mathbb M$ such that for every element $\phi\in\mathcal E_{\lambda}(\mathbb M)$ with unit $\operatorname{L}^2$-norm, the following inequality holds: \begin{equation}\label{2eq2}\Bigg\lvert\frac 1{|{\mathscr S}|}\sum_{\boldsymbol{x}\in {\mathscr S}}\phi(\boldsymbol{x})-\int \phi~d\sigma^{\mathbb M}\Bigg\rvert\leq \upsilon\end{equation}
\end{definition}

The main result in this section is a random construction that gives 
an approximate $(\lambda,2)$-design with arbitrarily high probability. 

\begin{theorem}\label{iterate3.1.1}
Let ${\mathscr S}:=\hat{\mathcal S}^\ell\boldsymbol{o}$ --- 
where ${\mathcal S}\subseteq\mathbb M$ is a random multisubset of isometries 
selected independently from the Haar measure on $\mathbb K$. 
Let 
$\delta\in (0,\frac 12)$. For any $\upsilon\in (0,1)$, 
and any integer $r>0$, if \begin{equation*}\label{seq0}|{{\mathcal S}}|
=16\ln 2\ln\begin{pmatrix}\frac{C_{\mathbb M}\lambda_r^{\frac {d_{\mathbb M}}2}}\delta\end{pmatrix}\,,\end{equation*} 
and \begin{align*}\label{0eqell1}\ell\geq \log_2\frac 1\upsilon+\log_2
C_{\mathbb M}+\frac {d_{\mathbb M}}4\log_2\lambda_r+\begin{pmatrix}d_{\mathbb M}-
\frac {\bar{d}_{\mathbb M}+1}2\end{pmatrix}\log_2\frac 1\epsilon\,,\end{align*} 
where $\epsilon= \lambda_r^{-\frac{d_{\mathbb M}+2}2}C_{\mathbb M}^{-\frac 14}
\upsilon^{\frac 12}$, then ${\mathscr S}\subseteq\mathbb M$ 
is an $\upsilon$-approximate $(\lambda_r,2)$-design with 
probability at least $1-2\delta$.
\end{theorem}

\begin{proof}
Let us fix an orthonormal basis for each of the eigenspaces $\mathcal H_\lambda$ on $\mathbb M$, say $(\phi_{\lambda,j})_{j=1}^{d_\lambda}$. Let ${\mathcal S}\subseteq\mathbb K$ be a finite nonempty multi-subset, consisting of independent Haar-distributed elements, and write $\hat{\mathcal S}={\mathcal S}\sqcup {\mathcal S}^{-1}$. We first take $\phi$ to be an element in this orthogonal basis, corresponding to eigenvalue $\lambda_\phi$ of the Casimir-Laplace-Beltrami $-\Delta_{\mathbb M}$, and aim to show that --- for the subset $\hat{\mathcal S}^\ell\boldsymbol{o}\subseteq\mathbb M$ --- the left hand side of \cref{2eq2} is small. Since \begin{equation*}\operatorname{H}_{\boldsymbol{o}}(\boldsymbol{x},\epsilon^2)=\sum_{\lambda\in\mathscr E_{\mathbb M}} e^{-\lambda \epsilon^2}\sum_{j\in [d_\lambda]}\phi_{\lambda,j}(\boldsymbol{o})\phi_{\lambda,j}(\boldsymbol{x})\,,\end{equation*} one has \begin{equation}\int_{\mathbb M}\phi(\boldsymbol{x})\operatorname{H}_{\boldsymbol{o}}(\boldsymbol{x},\epsilon^2)~d\sigma^{\mathbb M}(\boldsymbol{x})=e^{-\lambda_\phi\epsilon^2}\phi(\boldsymbol{o})\int_{\mathbb M}|\phi(\boldsymbol{x})|^2d\sigma^{\mathbb M}(\boldsymbol{x})=e^{-\lambda_\phi\epsilon^2}\phi(\boldsymbol{o})\,,\end{equation} from which it follows that 
\begin{align*}\Bigg\lvert\frac 1{|\hat{\mathcal S}|^\ell}\sum_{\boldsymbol{s}\in \hat{\mathcal S}^\ell}\phi(\boldsymbol{s}\boldsymbol{o})-\int \phi~d\sigma^{\mathbb M}\Bigg\rvert&=\Bigg\lvert e^{\lambda_\phi\epsilon^2}\frac 1{|\hat{\mathcal S}|^\ell}\sum_{\boldsymbol{s}\in \hat{\mathcal S}^\ell}\langle \phi,\operatorname{H}_{\boldsymbol{s}\boldsymbol{o}}\rangle-\int \phi~d\sigma^{\mathbb M}\Bigg\rvert\\ &\leq \Bigg\lvert\!\Bigg\lvert\sum_{\boldsymbol{s}\in \hat{\mathcal S}^\ell}\frac 1{|\hat{\mathcal S}|^\ell}e^{\lambda_\phi\epsilon^2}\operatorname{H}_{\boldsymbol{s}\boldsymbol{o}}-1_{\mathbb M}\Bigg\rvert\!\Bigg\rvert_{\operatorname{L}^2},\end{align*} where the last step is due to H\"{o}lder inequality. We aim to have \begin{displaymath}\Bigg\lvert\!\Bigg\lvert |\hat{\mathcal S}|^{-\ell}\sum_{\boldsymbol{s}\in \hat{\mathcal S}^\ell} e^{\lambda_\phi\epsilon^2}\operatorname{H}_{\boldsymbol{s}\boldsymbol{o}}-1_{\mathbb M}\Bigg\rvert\!\Bigg\rvert_{\operatorname{L}^2}\leq e^{\lambda_\phi\epsilon^2}~\Bigg\lvert\!\Bigg\lvert|\hat{\mathcal S}|^{-\ell}\sum_{\boldsymbol{s}\in \hat{\mathcal S}^\ell} \operatorname{H}_{\boldsymbol{s}\boldsymbol{o}}-1_{\mathbb M}\Bigg\rvert\!\Bigg\rvert_{\operatorname{L}^2}+e^{\lambda_\phi\epsilon^2}-1\end{displaymath} sufficiently small. Since \begin{equation}\label{appan}\lambda_\phi\epsilon^2\leq e^{\lambda_\phi\epsilon^2}-1\leq 2\lambda_\phi\epsilon^2\end{equation} if $0\leq \lambda_\phi\epsilon^2\leq \ln 2$, it suffices to have \begin{equation*}\max\left\{\Bigg\lvert\!\Bigg\lvert |\hat{\mathcal S}|^{-\ell}\sum_{\boldsymbol{s}\in \hat{\mathcal S}^\ell} \operatorname{H}_{\boldsymbol{s}\boldsymbol{o}}-1_{\mathbb M}\Bigg\rvert\!\Bigg\rvert_{\operatorname{L}^2},~2\lambda_\phi\epsilon^2\right\}\end{equation*} sufficiently small. 
In the general case, for an element $\phi\in\mathcal E_{\lambda_r}$ with unit $\operatorname{L}^2$-norm, say \begin{equation*}\phi=c_1\phi_{\lambda_{k_1},j_1}+\cdots+c_q\phi_{\lambda_{k_q},j_q}\,,\end{equation*} the above derivation takes the form \begin{align*}&\Bigg\lvert\frac 1{|\hat{\mathcal S}|^\ell}\sum_{\boldsymbol{s}\in \hat{\mathcal S}^\ell}\phi(\boldsymbol{s}\boldsymbol{o})-\int \phi~d\sigma^{\mathbb M}\Bigg\rvert\\ =~&\Bigg\lvert \sum_{a=1}^q c_a\left(e^{\lambda_{k_a}\epsilon^2}\frac 1{|\hat{\mathcal S}|^\ell}\sum_{\boldsymbol{s}\in \hat{\mathcal S}^\ell}\langle \phi_{\lambda_{k_a},j_a},\operatorname{H}_{\boldsymbol{s}\boldsymbol{o}}\rangle -\int_{\mathbb M} \phi_{\lambda_{k_a},j_a}d\sigma^{\mathbb M}\right)\Bigg\rvert\\ \leq~&\sum_{a=1}^q |c_a| \cdot\Bigg\lvert\!\Bigg\lvert\sum_{\boldsymbol{s}\in \hat{\mathcal S}^\ell}\frac 1{|\hat{\mathcal S}|^\ell}e^{\lambda_{k_i}\epsilon^2}\operatorname{H}_{\boldsymbol{s}\boldsymbol{o}}-1_{\mathbb M}\Bigg\rvert\!\Bigg\rvert_{\operatorname{L}^2}\end{align*} Recall that $\operatorname{H}^{\lambda_{r}}$ denotes the orthogonal projection of the heat kernel on $\mathcal E_{\lambda_r}$ for integer $r\geq 0$. For any $r_0\geq r$, we can replace $\operatorname{H}_{\boldsymbol{p}}$ by the truncated heat kernel $\operatorname{H}^{r_0}_{\boldsymbol{p}}$ in the above computations, and this gets us the following: \begin{align}\label{1ineq1}\Bigg\lvert\frac 1{|\hat{\mathcal S}|^\ell}\sum_{\boldsymbol{s}\in \hat{\mathcal S}^\ell}\phi(\boldsymbol{s}\boldsymbol{o})-\int \phi~d\sigma^{\mathbb M}\Bigg\rvert 
& \leq \sum_{a=1}^q |c_a| \cdot\Bigg\lvert\!\Bigg\lvert\sum_{\boldsymbol{s}\in \hat{\mathcal S}^\ell}\frac 1{|\hat{\mathcal S}|^\ell}e^{\lambda_{k_a}\epsilon^2}\operatorname{H}^{\lambda_{r_0}}_{\boldsymbol{s}\boldsymbol{o}}-1_{\mathbb M}\Bigg\rvert\!\Bigg\rvert_{\operatorname{L}^2}\nonumber
\\ &\leq \Bigg\lvert\!\Bigg\lvert\sum_{\boldsymbol{s}\in \hat{\mathcal S}^\ell}\frac 1{|\hat{\mathcal S}|^\ell}e^{\lambda_{r}\epsilon^2}\operatorname{H}^{\lambda_{r_0}}_{\boldsymbol{s}\boldsymbol{o}}-1_{\mathbb M}\Bigg\rvert\!\Bigg\rvert_{\operatorname{L}^2} \cdot \sum_{a=1}^q |c_a|\nonumber\\ 
&\leq \sqrt s\left(e^{\lambda_{r}\epsilon^2}\cdot\Bigg\lvert\!\Bigg\lvert\sum_{\boldsymbol{s}\in \hat{\mathcal S}^\ell}\frac 1{|\hat{\mathcal S}|^\ell}\operatorname{H}^{\lambda_{r_0}}_{\boldsymbol{s}\boldsymbol{o}}-1_{\mathbb M}\Bigg\rvert\!\Bigg\rvert_{\operatorname{L}^2}+2\lambda_r\epsilon^2\right)\,,\end{align} where the last step is due to an application of Cauchy-Schwartz inequality and the inequality \cref{appan}. By \cref{dime}, 
we have \begin{equation*}s\leq \dim \mathcal E_{\lambda_r}=C_{\mathbb M}\lambda_{r}^{\frac{d_{\mathbb M}}2}\end{equation*} Given $\upsilon\in (0,1)$, suppose that we chose $\epsilon\in (0,2^{-e})$ small and $\ell>0$ large, and also $r_0\geq r$ appropriately, so that \begin{equation}\label{maxineq1}\max\left\{{\lambda_{r}\epsilon^2}, \Bigg\lvert\!\Bigg\lvert\sum_{{\boldsymbol{s}}\in \hat{\mathcal S}^\ell}\frac 1{|\hat{\mathcal S}|^\ell}\operatorname{H}^{\lambda_{r_0}}_{\boldsymbol{s}\boldsymbol{o}}-1_{\mathbb M}\Bigg\rvert\!\Bigg\rvert_{\operatorname{L}^2}\right\}\leq 
\frac{\upsilon}{\lambda_r^{\frac {d_{\mathbb M}}4}C_{\mathbb M}^{\frac 12}}\,.\end{equation} From \cref{1ineq1}, it then follows that $\hat{\mathcal S}^\ell\boldsymbol{o}\subseteq\mathbb M$ is a $3\upsilon$-approximate design. Now, by 
\cref{1des1}, for any $\epsilon\in (0,1)$ and any integer $r_0\geq 0$, we have \begin{equation*}\Bigg\lvert\!\Bigg\lvert
\operatorname{H}^{\lambda_{r_0}}_{\boldsymbol{s}\boldsymbol{o}}\Bigg\rvert\!\Bigg \rvert_{\operatorname{L}^2} \leq \sqrt{C_{\mathbb M}{\mathfrak v}_{\mathbb M}2^{-d_{\mathbb M}+\frac{\bar{d}_{\mathbb M}+1}2}}\epsilon^{-d_{\mathbb M}+\frac{\bar{d}_{\mathbb M}+1}2}\,.\end{equation*} And, by inequality \cref{iterate}, we have \begin{align*}
\label{eqdisteq01}\operatornamewithlimits{\mathbb P}_{\mathcal S}\left({\Bigg|\!\Bigg|{\frac 1{|{\hat {\mathcal S}}|^\ell}}\sum_{\boldsymbol{s}\in {\hat {\mathcal S}}^\ell} {\operatorname{H
}}^{\lambda_{r_0}}_{\boldsymbol{s}\boldsymbol{o}}-1_{\mathbb M}\Bigg|\!\Bigg|_{\operatorname{L}^2}\leq 
{\frac 1{2^\ell}}\Bigg|\!\Bigg|{\operatorname{H}}^{\lambda_{r_0}}_{\boldsymbol{o}}
\Bigg|\!\Bigg|_{\operatorname{L}^2}}\right)\geq 1-{\frac {2C_{\mathbb M}}
{\eta}}\exp\left({\frac{-|{{\mathcal S}}|}{16\ln 2}}\right)\,,\end{align*} with 
$\eta=\lambda_{r_0}^{-\frac {d_{\mathbb M}}2}$.
From these, it follows that \begin{align*}1-{\frac {2C_{\mathbb M}}
{\eta}}\exp\left({\frac{-|{{\mathcal S}}|}{16\ln 2}}\right)&\leq~\operatornamewithlimits{\mathbb P}_{\mathcal S}\left({\Bigg|\!\Bigg|{\frac 1{|{\hat {\mathcal S}}|^\ell}}\sum_{\boldsymbol{s}\in {\hat {\mathcal S}}^\ell} {\operatorname{H}}^{\lambda_{r_0}}_{\boldsymbol{s}\boldsymbol{o}}-1_{\mathbb M}\Bigg|\!\Bigg|_{\operatorname{L}^2}\leq {\frac 1{2^\ell}}\Bigg|\!\Bigg|{\operatorname{H}}^{\lambda_{r_0}}_{\boldsymbol{o}}
\Bigg|\!\Bigg|_{\operatorname{L}^2}}\right)\\ &\leq~\operatornamewithlimits{\mathbb P}_{\mathcal S}\left({\Bigg|\!\Bigg|{\frac 1{|{\hat {\mathcal S}}|^\ell}}\sum_{\boldsymbol{s}\in {\hat {\mathcal S}}^\ell} {\operatorname{H}}^{\lambda_{r_0}}_{\boldsymbol{s}\boldsymbol{o}}-1_{\mathbb M}\Bigg|\!\Bigg|_{\operatorname{L}^2}\leq {\frac 1{2^\ell}} \sqrt{C_{\mathbb M}{\mathfrak v}_{\mathbb M}}\epsilon^{-d_{\mathbb M}+\frac{\bar{d}_{\mathbb M}+1}2}}\right)\end{align*} From \cref{maxineq1}, it is imperative that we select 
$\epsilon\in (0,2^{-e})$ such that the following holds: \begin{equation}
\epsilon= \lambda_r^{-\frac{d_{\mathbb M}+2}2}C_{\mathbb M}^{-\frac 14}
\upsilon^{\frac 12}\,.\end{equation}
Suppose, moreover, that we select integer $\ell> 0$ that satisfies \begin{align}
\label{eqell1}\ell\geq \log_2
\frac 1\upsilon+\log_2C_{\mathbb M}+\frac {d_{\mathbb M}}4\log_2\lambda_r+
\left(d_{\mathbb M}-\frac {\bar{d}_{\mathbb M}+1}2\right)\log_2\frac 1\epsilon\end{align}
Then, \cref{maxineq1} is guaranteed. Finally, we set $r_0=r$, and 
for any $\delta\in (0,\frac 12)$, we set \begin{equation}\label{seq}|{{\mathcal S}}|
=16\ln 2\ln\left(\frac{C_{\mathbb M}\lambda_r^{\frac {d_{\mathbb M}}2}}\delta\right)\,,\end{equation} 
and the discussion above shows that --- with probability at least $1-2\delta$ --- 
the inequality \cref{2eq2} is satisfied if we set ${\mathscr S}:=
\hat{\mathcal S}^\ell\boldsymbol{o}$ --- where ${\mathcal S}\subseteq
\mathbb M$ is a random multisubset of isometries, whose 
elements are choosen independently from Haar measure on $\mathbb K$ --- and 
demand that \cref{eqell1} and \cref{seq} hold, as desired.
\end{proof}


\medskip

\subsection{Bound on Wasserstein distance}\label{ssec:2}
In this subsection, we assume that we have oracle access to function 
values at queried points. we start with briefly recalling the 
following notion of quadratic exposedness of submanifolds of $\mathbb R^n$, 
first appeared in \cite{FILN}.

\begin{definition}
A compact submanifold $M\subseteq \mathbb R^n$ is said to be quadratically 
$R_m$-exposed at a point $m\in M$ for some $R_m>0$ if there is
an ${\boldsymbol{x}}_m\in\mathbb R^n$ such that $M\subseteq B_{R_m}
(\boldsymbol{x}_m,\mathbb R^n)$ and $m\in \partial B_{R_m}
(\boldsymbol{x}_m,\mathbb R^n)\cap M$. If $M$ is quadratically $R_m$-exposed 
at each $m\in M$, then we say that $M$ is uniformly quadratically exposed.
\end{definition}

\begin{remark}
Any submanifold of euclidean sphere is uniformly quadratically exposed. Torus $(\mathbb S^1)^d$
is not uniformly quadratically exposed for $d>1$.
\end{remark}

\begin{remark}\label{rotate}
Let $\partial_{\mathbb M}(\cdot,\cdot)$ denote the Riemannian distance on $\mathbb M$. 
By Nash' embedding theorem, there is an isometric 
embedding \begin{displaymath}\chi:(\mathbb M,~\partial_{\mathbb M})\rightarrow
(\mathbb R^{d_{\mathbb M}^2+d_{\mathbb M}}, |\!|~|\!|_2)\,.
\end{displaymath} Let $R>0$ be such that 
a closed ball $\mathbb B_R(\boldsymbol{x}_0)$ 
of radius $R>0$ and center $\boldsymbol{x}_0\in
\mathbb R^{d_{\mathbb M}^2+d_{\mathbb M}}$
satisfies $\chi(\mathbb M)\subseteq \mathbb B_R(\boldsymbol{0})$. 
By decreasing $R$ continuously, 
we can find an $R_0$ such that $\chi(\mathbb M)\subseteq 
\mathbb B_{R_0}(\boldsymbol{x}_0)$ and $\chi(\mathbb M)\cap 
\mathbb S_{R_0}(\boldsymbol{x}_0)\neq\emptyset$, where $\mathbb S_{R_0}(\boldsymbol{x}_0)$ 
denotes the sphere of radius $R_0$ and center at $\boldsymbol{x}_0
\in\mathbb R^{d_{\mathbb M}^2+d_{\mathbb M}}$. 
Fix a point $\boldsymbol{p}=\chi(\boldsymbol{m}_0)
\in \chi(\mathbb M)\cap\partial \mathbb B_{R_0}(\boldsymbol{0})$. 
For any $\boldsymbol{m}\in \mathbb M$, let $\sigma^{\mathbb M}:[-1,1]\rightarrow\mathbb M$ 
be a geodesic such that $\sigma^{\mathbb M}(-1)=\boldsymbol{m}$ and 
$\sigma^{\mathbb M}(1)=\boldsymbol{m}_0$. Let $\zeta:\mathbb M\rightarrow \mathbb M$ 
be an isometry that satisfies $\zeta(\sigma^{\mathbb M}(t))=\sigma^{\mathbb M}(-t)$; then 
$(\chi\circ\zeta)(\mathbb M)=\chi(\mathbb M)\subseteq \mathbb B_{R_0}
(\boldsymbol{x}_0)$, and $\boldsymbol{p}=(\chi\circ\zeta)(\boldsymbol{m})$. 
Thus, for any point $\boldsymbol{m}\in\mathbb M$, there is an isometric embedding 
\begin{displaymath}\chi_{\boldsymbol{m}}:(\mathbb M,~\partial_{\mathbb M})\rightarrow
(\mathbb R^{d_{\mathbb M}^2+d_{\mathbb M}}, |\!|~|\!|_2)\,,\end{displaymath} such that 
$\chi_{\boldsymbol{m}}(\mathbb M)$ is quadratically 
exposed at $\chi_{\boldsymbol{m}}({\boldsymbol{m}})$. 
\end{remark}

\begin{lemma}\label{vari}
Let $\sigma^{\mathbb M}$ the probability measure on $\mathbb M$ 
induced by Haar measure on $\mathbb K$. For any fixed 
$\boldsymbol{x}_0\in\mathbb M$, and any $t > 0$ 
be a random variable on $\mathbb M$ with density $\operatorname{H}_{
\boldsymbol{x}_0}(\boldsymbol{x},t)$, and let ${\boldsymbol{x}}_t:
=\partial_{\mathbb M}(\boldsymbol{x}_0,\boldsymbol{X}_t)$. Then 
\begin{equation}\label{exp-distance}\mathbb E[{\boldsymbol{x}}_t^2]\leq 
d_{\mathbb M}t\end{equation}
\end{lemma}

\begin{proof}[Proof of \cref{vari}]
Fornotational simplicity, we write $\operatorname{H}_t(\boldsymbol{x}):
=\operatorname{H}_{\boldsymbol{x}_0}(
\boldsymbol{x},t)$ for the time-$t$ heat kernel on $\mathbb M$, 
and let $\nu_t^\ast$ be the Borel measure corresponding to 
the top form 
\[d\nu_t^\ast=\operatorname{H}_t(\boldsymbol{x})~d\sigma^{\mathbb M}\] 
Let $\{\boldsymbol{x}_u\mid u\in [0,t]\}$ be the Brownian motion in 
$\mathbb M$, starting at $\boldsymbol{x}_0\in\mathbb M$, with infinitesimal 
generator \begin{equation*}{\mathcal H}_{\mathbb M}:
=\frac d{dt}-\frac 12\Delta_{\mathbb M}\,.\end{equation*}
For each positive integer $m>0$, consider the equipartition 
\[0=t_0<t_1<\cdots<t_m=t\] where $t_{j+1}-t_j=m^{-1}t$ for $j\in\{0,1,\cdots,m-1\}$. 
Define $\{\boldsymbol{x}_j^{(m)}\}_{j=0}^m$ 
such that $\boldsymbol{x}_j^{(m)}=\boldsymbol{x}_{\frac {jt}m}$. 
As stated in remark \cref{rotate}, we fix a closed isometric 
embedding \begin{displaymath}\chi_{\boldsymbol{x}_{j-1}^{(m)}}:
(\mathbb M,~\partial_{\mathbb M})\rightarrow
(\mathbb R^{d_{\mathbb M}^2+d_{\mathbb M}}, |\!|~|\!|_2)
\,,\end{displaymath} such that $\chi_{\boldsymbol{x}_{j-1}^{(m)}}
(\mathbb M)\subseteq \mathbb S_{R_0}(\boldsymbol{a})$, 
and \begin{displaymath}\boldsymbol{x}^{(m)}_{j-1}\in\chi_{\boldsymbol{x}_{j-1}^{(m)}}
(\mathbb M)\cap \mathbb S_{R_0}(\boldsymbol{a})\,,\end{displaymath} ensuring 
quadratic exposedness at $\chi_{\boldsymbol{x}_{j-1}^{(m)}}
({\boldsymbol{x}_{j-1}^{(m)}})$. Moreover, by applying an affine 
transformation of $\mathbb R^{d_{\mathbb M}^2+d_{\mathbb M}}$, we 
may assume that $\boldsymbol{x}_0=\boldsymbol{0}$.\\

We now identify $\mathbb M$ with $\chi_{\boldsymbol{x}_{j-1}^{(m)}}
(\mathbb M)$. Note that \begin{align}\label{first1}\partial_{\mathbb M}(\boldsymbol{x}_t,
\boldsymbol{x}_0)^2&=|\!|\boldsymbol{x}_t-\boldsymbol{x}_0
|\!|_2^2\nonumber\\ &=\sum_{j=1}^m|\!|\boldsymbol{x}_{j}^{(m)}
-\boldsymbol{x}_{{j-1}}^{(m)}|\!|_2^2+2\sum_{j=1}^m
\langle\boldsymbol{x}_{j}^{(m)}-\boldsymbol{x}_{{j-1}}^{(m)},
~\boldsymbol{x}_{{j-1}}^{(m)}\rangle\nonumber\\ &=\sum_{j=1}^m
\partial_{\mathbb M}(\boldsymbol{x}_j^{(m)}, \boldsymbol{x}_{{j-1}}^{(m)})^2
+2\sum_{j=1}^m \langle\boldsymbol{x}_{j}^{(m)}-\boldsymbol{x}_{{j-1}}^{(m)},
~\boldsymbol{x}_{{j-1}}^{(m)}\rangle\,,\end{align} Suppose 
that $\boldsymbol{v}_{j-1}^{(m)}\in T_{\boldsymbol{x}_{j-1}^{(m)}}
(\mathbb M)$ is a unit normed tangent vector. 
Write $\mathbb S_{R_0}:=\mathbb S_{R_0}
(\boldsymbol{a})$, and let $\Pi$ be the orthogonal projection 
onto $T_{\boldsymbol{x}_{j-1}^{(m)}}(\mathbb S_{R_0})$. 
For any $\epsilon_0>0$ sufficiently small, 
Federer's tangency condition (see \cite{MR0257325}) implies that there is 
$\boldsymbol{p}_{\epsilon_0}\in\mathbb M$ such that $|\!|
\boldsymbol{p}-\boldsymbol{x}^{(m)}_{j-1}|\!|_2<\epsilon_0$, and 
\begin{displaymath}\Bigg\lvert\!\Bigg\lvert\boldsymbol{v}-\frac{\boldsymbol{p}
-\boldsymbol{x}_{j-1}^{(m)}}{|\!|\boldsymbol{p}-
\boldsymbol{x}_{j-1}^{(m)}|\!|_2}\Bigg\rvert\!\Bigg\rvert_2<\epsilon_0\,.\end{displaymath} 
If $\boldsymbol{q}\in \mathbb S_{R_0}$ is the point of intersection of the line segment 
joining $\boldsymbol{p}$ and $\Pi(\boldsymbol{p})$, it is easy to 
verify that $|\!|\boldsymbol{q}-\boldsymbol{x}^{(m)}_{j-1}|\!|_2<\epsilon_0$, and
\begin{align*}\Bigg\lvert\!\Bigg\lvert\boldsymbol{v}-\frac{\boldsymbol{q}
-\boldsymbol{x}_{j-1}^{(m)}}{|\!|\boldsymbol{q}-\boldsymbol{x}_{j-1}^{(m)}
|\!|_2}\Bigg\rvert\!\Bigg\rvert_2<\epsilon_0\,.\end{align*} In particular, 
$T_{\boldsymbol{x}_{j-1}^{(m)}}(\mathbb M)\subseteq T_{\boldsymbol{x}_{j-1}^{(m)}}
\mathbb S_{R_0}$.\\

For each $j\in [m]$, let $T_{j-1}^{(m)}$ denote the affine tangent space 
$T_{\boldsymbol{x}_{j-1}^{(m)}}(\mathbb S_{R_0})$ at the point $\boldsymbol{x}_{j-1}^{(m)}$. 
Write \begin{equation}\label{someq}\boldsymbol{z}_{j-1}^{(m)}:=
\operatornamewithlimits{argmin}_{\boldsymbol{u}\in T_{j-1}^{(m)}}
|\!|\boldsymbol{u}|\!|_2,\hspace{0.5cm}\boldsymbol{p}_{j-1}^{(m)}:=
\operatornamewithlimits{argmin}_{\boldsymbol{u}\in T_{j-1}^{(m)}}
|\!|\boldsymbol{u}-\boldsymbol{x}_j^{(m)}|\!|_2\,.\end{equation} 
One has the following orthogonality relations: 
\begin{equation}\label{someq1}\boldsymbol{p}_{j-1}^{(m)}-\boldsymbol{x}_{j-1}^{(m)}~
\bot~\boldsymbol{z}_{j-1}^{(m)}, \hspace{1cm}\operatorname{and}\hspace{1cm}
\boldsymbol{x}_j^{(m)}-\boldsymbol{p}_{j-1}^{(m)}~
\bot~\boldsymbol{x}_{j-1}^{(m)}-\boldsymbol{z}_{j-1}^{(m)}\,.\end{equation} 

\vspace{0.25cm}

\begin{claim}\label{claim-negative1}
\begin{equation}\label{angle1}\langle \boldsymbol{x}_j^{(m)}-\boldsymbol{x}_{j-1}^{(m)},
\boldsymbol{x}_{j-1}^{(m)}\rangle\leq 0\end{equation}
\end{claim}

\begin{proof}[Proof of \cref{claim-negative1}]
Let $\boldsymbol{q}_0^{(m)}\in \mathbb S_{R_0}$ be the point where 
$\mathbb S_{R_0}$ intersects the line segment joining 
$\boldsymbol{x}_0^{(m)}$ and $\boldsymbol{z}_{j-1}^{(m)}$; 
similarly, let $\boldsymbol{q}_j^{(m)}\in \mathbb S_{R_0}$ be the point where 
$\mathbb S_{R_0}$ intersects the line segment joining 
$\boldsymbol{x}_j^{(m)}$ and $\boldsymbol{p}_{j-1}^{(m)}$. 
By quadratic exposedness, one has \begin{align*}\operatorname{sgn}
\langle \boldsymbol{x}_j^{(m)}-\boldsymbol{p}_{j-1}^{(m)},~
\boldsymbol{z}_{j-1}^{(m)}\rangle&=\operatorname{sgn}
\langle \boldsymbol{x}_j^{(m)}-\boldsymbol{p}_{j-1}^{(m)},~
\boldsymbol{z}_{j-1}^{(m)}\rangle\\ &=\operatorname{sgn}
\langle \boldsymbol{q}_j^{(m)}-\boldsymbol{p}_{j-1}^{(m)},~
\boldsymbol{z}_{j-1}^{(m)}-\boldsymbol{q}_0^{(m)}\rangle\end{align*} 
Let ${\bf n}$ be the unit normal to $T_{j-1}^{(m)}$ pointing inward 
of $\mathbb S_{R_0}$; then the orthogonality relations 
\begin{displaymath}\boldsymbol{q}_j^{(m)}-\boldsymbol{p}_{j-1}^{(m)}
|\!|~\boldsymbol{n},\hspace{0.5cm}\boldsymbol{z}_{j-1}^{(m)}
-\boldsymbol{q}_0^{(m)}|\!| -\boldsymbol{n}\,,\end{displaymath} which proves the claim. 
\end{proof} 

\begin{claim}\label{aclaim}
\begin{equation}\mathbb E
\left[\langle\boldsymbol{x}_{j}^{(m)}
-\boldsymbol{x}_{{j-1}}^{(m)},~\boldsymbol{x}_{{j-1}}^{(m)}\rangle\right] \leq 0
\end{equation}
\end{claim}

\begin{proof}[Proof of \cref{aclaim}]
Note that \begin{align*}\langle\boldsymbol{x}_{j}^{(m)}
-\boldsymbol{x}_{{j-1}}^{(m)}, ~\boldsymbol{x}_{{j-1}}^{(m)}\rangle&= 
\langle\boldsymbol{p}_{j-1}^{(m)}-\boldsymbol{x}_{{j-1}}^{(m)},
~\boldsymbol{z}_{{j-1}}^{(m)}\rangle+\langle\boldsymbol{x}_{j}^{(m)}
-\boldsymbol{p}_{{j-1}}^{(m)},~\boldsymbol{x}_{{j-1}}^{(m)}-
\boldsymbol{z}_{{j-1}}^{(m)}\rangle\\ &\hspace{1cm} + \langle\boldsymbol{x}_{j}^{(m)}
-\boldsymbol{p}_{{j-1}}^{(m)},~\boldsymbol{z}_{{j-1}}^{(m)}
\rangle+\langle\boldsymbol{p}_{j-1}^{(m)}
-\boldsymbol{x}_{{j-1}}^{(m)},~\boldsymbol{x}_{{j-1}}^{(m)}-
\boldsymbol{z}_{{j-1}}^{(m)}\rangle\\
&= \langle\boldsymbol{x}_{j}^{(m)}
-\boldsymbol{p}_{{j-1}}^{(m)},~\boldsymbol{z}_{{j-1}}^{(m)}
\rangle+\langle\boldsymbol{p}_{j-1}^{(m)}
-\boldsymbol{x}_{{j-1}}^{(m)},~\boldsymbol{x}_{{j-1}}^{(m)}-
\boldsymbol{z}_{{j-1}}^{(m)}\rangle\\ &\leq
\langle\boldsymbol{p}_{j-1}^{(m)}
-\boldsymbol{x}_{{j-1}}^{(m)},~\boldsymbol{x}_{{j-1}}^{(m)}-
\boldsymbol{z}_{{j-1}}^{(m)}\rangle\,,\end{align*} where 
the last two steps follow due to \cref{someq1} and \cref{angle1}. 
An application of dominated convergence 
theorem (due to compactness of $\mathbb M$) implies 
\begin{align}\label{onemore1}&\lim_{m\rightarrow\infty}\mathbb E
\left[\langle\boldsymbol{p}_{j-1}^{(m)}
-\boldsymbol{x}_{{j-1}}^{(m)},~\boldsymbol{x}_{{j-1}}^{(m)}-
\boldsymbol{z}_{{j-1}}^{(m)}\rangle\right] \nonumber\\ =~&
\mathbb E\left[\lim_{m\rightarrow\infty}\langle\boldsymbol{p}_{j-1}^{(m)}
-\boldsymbol{x}_{{j-1}}^{(m)},~\boldsymbol{x}_{{j-1}}^{(m)}-
\boldsymbol{z}_{{j-1}}^{(m)}\rangle\right]\nonumber\\
=~&0\,,\end{align} which establishes the claim.
\end{proof}

\begin{claim}\label{stereo}For $j\in [m]$, let \begin{equation}s_j:=
\mathbb E^{\boldsymbol{x}_0}\left[\partial_{\mathbb M}\left(\boldsymbol{x}_{t_j}^{(m)},
\boldsymbol{x}_{t_{j-1}}^{(m)}\right)^2\right]\,;\end{equation} then 
\begin{equation}\label{liminfgen}\lim_{m\rightarrow\infty}
\sum_{j=1}^m s_j \leq \frac 12d_{\mathbb M}t\,.\end{equation}
\end{claim}

\begin{proof}[Proof of \cref{stereo}]
Define $u(t,\boldsymbol{x}):=\mathbb E^{\boldsymbol{x}}[\partial_{\mathbb M}(
\boldsymbol{x}, \boldsymbol{x}_t)^2]$; then $u(t,\boldsymbol{x})$ solves 
\cite[theorem 4.2.1]{MR1882015} the initial-boundary 
value problem \begin{equation}{\mathcal H}_{\mathbb M}u=0,\hspace{0.5cm}
\lim_{t\rightarrow 0^+}u(t,\boldsymbol{x})=0\end{equation} In particular, one has 
\begin{align*}\lim_{t\rightarrow 0^+}\frac{u(t,\boldsymbol{x})}
t &=\frac d{dt}u(t,\boldsymbol{x})\Bigg\rvert_{t=0^+}\\ &=
\frac12\Delta_{\mathbb M} \partial_{\mathbb M}(
\boldsymbol{x}, \boldsymbol{x}_t)^2
\Bigg\rvert_{\boldsymbol{x}=\boldsymbol{x}_t}\end{align*} 
An application of local Taylor expansion of the Riemannian 
metric on $\mathbb M$ (see \cite{MR2024928}) implies 
\begin{equation}\label{laplace-distance}\Delta_{\mathbb M} \operatorname{dist}(
\boldsymbol{x}, \boldsymbol{x}_t)^2
\Bigg\rvert_{\boldsymbol{x}=\boldsymbol{x}_t}\leq d_{\mathbb M}\,.\end{equation} 
Since \begin{align*}s_j&=\mathbb E^{\boldsymbol{x}_0}
\left[\partial_{\mathbb M}\left(\boldsymbol{x}_{t_j}^{(m)},
\boldsymbol{x}_{t_{j-1}}^{(m)}\right)^2\right]\\ &= 
\mathbb E^{\boldsymbol{x}_0}
\left[\partial_{\mathbb M}\left(\boldsymbol{x}_{\frac tm},
\boldsymbol{x}_0\right)^2\right]\end{align*} for all $j\in [m]$, 
we conclude \begin{align*}\lim_{m\rightarrow\infty}
\sum_{j=1}^m s_j&= t\lim_{m\rightarrow\infty}
\frac{\mathbb E^{\boldsymbol{x}_0}
\left[\partial_{\mathbb M}\left(\boldsymbol{x}_{\frac tm},
\boldsymbol{x}_0\right)^2\right]}{\frac tm} \\ &\leq \frac 12
d_{\mathbb M}t\,,\end{align*} finishing the proof of the claim.
\end{proof}

Finally, to finish the proof of \cref{vari}, notice that \begin{align*}
\mathbb E[{\boldsymbol{x}}_t^2]&=\int_{\mathbb M}\partial_{\mathbb M}(\boldsymbol{x}_0,
\boldsymbol{x})^2\operatorname{H}_t(\boldsymbol{x})~\sigma^{\mathbb M}(d\boldsymbol{x})\\
&=\mathbb E^{\boldsymbol{x}_0}[\partial_{\mathbb M}
(\boldsymbol{x}_0,\boldsymbol{x}_t)^2]\,,\end{align*} where 
the last expectation is over a Brownian motion $\{
\boldsymbol{x}_t; t\geq 0\}$ starting at $\boldsymbol{x}_0$, 
with infinitesimal generator ${\mathcal H}_{\mathbb M}$ as in 
\cref{gen0}. The lemma now follows by \cref{first1}, 
\cref{onemore1}, and \cref{liminfgen}.
\end{proof}

To proceed further, we will need to apply the following duality theorem by Kantorovi\v{c} and Rubin\v{s}te\'{i}n, relating the 1-Wasserstein distance to integration of 1-Lipschitz functions; see \cite{MR0102006} for more details and a proof.

\begin{theorem}[Kantorovi\v{c} - Rubin\v{s}te\'{i}n]\label{RubiK}
Let $(\mathbb M,\partial_{\mathbb M})$ be a compact connected Riemannian manifold, and ${\mathscr P}(\mathbb M)$ the space of probability measures on $\mathbb M$ --- equipped with the 1-Wasserstein distance $W_1(\cdot,\cdot)$. For any $\sigma^{\mathbb M},\nu\in {\mathscr P}(\mathbb M)$, the following equality holds: \begin{equation}\label{Rub}W_1(\mu,\nu)=\sup_{\phi\in \operatorname{Lip}_1(\mathbb M)}\begin{pmatrix}\int_{\mathbb M}\phi~d\sigma^{\mathbb M}-\int_{\mathbb M}\phi~d\nu\end{pmatrix}\end{equation}
\end{theorem}



We are now ready to formulate our result on equidistribution of the cover constructed above.

\begin{theorem}\label{Wassers}
Let ${\mathcal S}\subset \mathbb K$ be a nonempty finite multisubset whose 
elements are selected independently at random from the Haar measure 
on $\mathbb K$. Let $\delta\in (0,\frac 12)$, and assume that $
\epsilon\in (0,2^{-e})$ is sufficiently small. Suppose that 
the cardinality of ${\mathcal S}$ satisfies \begin{equation}\label{010seq1}
|{\mathcal S}|= 16\ln 2\left(\ln C_{\mathbb M} 
+\frac {d_{\mathbb M}}4\log_2\begin{pmatrix}
{\frac{d_{\mathbb M}}{2\epsilon^2}}\end{pmatrix}
+\frac{d_{\mathbb M}}2\log_2\log_2\begin{pmatrix}{\frac{d_{\mathbb M}}
{2\epsilon^2}}\end{pmatrix}+\ln \frac 1{e\delta}\right)\,.\end{equation} 
Let $\ell>0$ be an integer satisfying \begin{equation}\label{010seq01}
\ell\geq \begin{pmatrix}d_{\mathbb M}-\frac{\bar{d}_{\mathbb M}+1}2\end{pmatrix}
\log_2\frac 1\epsilon+\log_2\begin{pmatrix}\frac{6C_{\mathbb M}}
{{\mathfrak v}_{\mathbb M}}\end{pmatrix}+\frac {d_{\mathbb M}}4\log_2
\begin{pmatrix}\frac 1{\pi r^2_{\epsilon,\mathbb M}}\end{pmatrix}
+\frac 12\log_2\Gamma\begin{pmatrix}{\frac {d_{\mathbb M}}2}+1
\end{pmatrix}\,.\end{equation} 
Fix a $\boldsymbol{p}_0\in\mathbb M$, and let $\nu$ be the empirical 
measure supported on the random multisubset 
$\mathscr S:=\hat{\mathcal S}^\ell\boldsymbol{p}_0\subseteq \mathbb M$; 
if $\epsilon$ is sufficiently small, then 
\begin{displaymath}W_1(\sigma^{\mathbb M},\nu)\leq 
2\sqrt{d_{\mathbb M}\epsilon}\,.\end{displaymath} 
\end{theorem}

\begin{proof}
Let $\operatorname{Lip}_{1,0}(\mathbb M)$ be the set of mean-zero $\operatorname{Lip}_1$-functions on $\mathbb M$. By theorem \cref{RubiK}, it suffices to show that \begin{equation}\sup_{\phi\in \operatorname{Lip}_{1,0}(\mathbb M)}\begin{pmatrix} \int_{\mathbb M}\phi~d\sigma^{\mathbb M}-\int_{\mathbb M}\phi~d\nu\end{pmatrix}<2\sqrt{d_{\mathbb M}}\epsilon\end{equation} We note that, for any such function $\phi\in \operatorname{Lip}_{1,0}(\mathbb M)$, if $\phi(\boldsymbol{x}_0)=|\!|\phi|\!|_{L^\infty}$ then \begin{align*}0&=\int_{\mathbb M}\phi(\boldsymbol{x})~d\sigma^{\mathbb M}(\boldsymbol{x}) \\ &= \int_{\mathbb M}\phi(\boldsymbol{x}_0)~d\sigma^{\mathbb M}(\boldsymbol{x}) + \int_{\mathbb M}\begin{pmatrix}\phi(\boldsymbol{x})-\phi(\boldsymbol{x}_0)\end{pmatrix}~d\sigma^{\mathbb M}(\boldsymbol{x}) \\ &=\phi(\boldsymbol{x}_0)+\int_{\mathbb M}\begin{pmatrix}\phi(\boldsymbol{x})-\phi(\boldsymbol{x}_0)\end{pmatrix}~d\sigma^{\mathbb M}(\boldsymbol{x})\\ \Rightarrow\hspace{0.5cm}\phi(\boldsymbol{x}_0)&\leq \int_{\mathbb M}|\phi(\boldsymbol{x})-\phi(\boldsymbol{x}_0)| ~d\sigma^{\mathbb M}(\boldsymbol{x})\\ &\leq \int_{\mathbb M}\partial_{\mathbb M}(\boldsymbol{x},\boldsymbol{x}_0)~d\sigma^{\mathbb M}(\boldsymbol{x})\\ &\leq \operatorname{diam}(\mathbb M)\end{align*} For sufficiently small $t>0$, we let $\nu_t^\ast$ be the Borel probability measure on $\mathbb M$ whose density is \begin{equation*}{\frac {d\nu^\ast_t(\boldsymbol{x})}{d\sigma^{\mathbb M}}}={\frac 1{|\hat{\mathcal S}^\ell|}}\sum_{\boldsymbol{p}\in {\hat{\mathcal S}}^\ell\boldsymbol{p}_0}\operatorname{H}_{\boldsymbol{p}}(\boldsymbol{x},t)\end{equation*} Then, for $t={\epsilon^2}$, one has \begin{align}W_1(\sigma^{\mathbb M},\nu^\ast_t)& =  \sup_{\phi\in\operatorname{Lip}_{1,0}(\mathbb M) }\begin{vmatrix}\int_{\mathbb M}\phi(\boldsymbol{x})~d\sigma^{\mathbb M}(\boldsymbol{x})-\int_{\mathbb M}\phi(\boldsymbol{x})~d\nu^\ast_t(\boldsymbol{x})\end{vmatrix}\nonumber\\ &\leq \sup_{\phi\in\operatorname{Lip}_{1,0}(\mathbb M)}\int_{\mathbb M}|\phi(\boldsymbol{x})|\cdot \begin{vmatrix}1_{\mathbb M}-{\frac 1{|\hat{\mathcal S}|^\ell}}\sum_{\boldsymbol{p}\in {\hat{\mathcal S}}^\ell\boldsymbol{p}_0}\operatorname{H}_{\boldsymbol{p}}(\boldsymbol{x},t)\end{vmatrix}d\sigma^{\mathbb M}(\boldsymbol{x})\nonumber\\ \hspace{1cm}&\leq \sup_{\phi\in\operatorname{Lip}_{1,0}(\mathbb M) } \begin{Vmatrix}\phi\end{Vmatrix}_{L^\infty}\cdot \int_{\mathbb M}\begin{vmatrix}1_{\mathbb M}-{\frac 1{|\hat{\mathcal S}|^\ell}}\sum_{y \in {\hat{\mathcal S}}^\ell\boldsymbol{p}_0}\operatorname{H}_{\boldsymbol{p}}(\boldsymbol{x},t)\end{vmatrix}d\sigma^{\mathbb M}(\boldsymbol{x})\nonumber\\ \label{1stw}&\leq \operatorname{diam}(\mathbb M)\epsilon^{\frac{d_{\mathbb M}}2}d_{\mathbb M}^{-\frac{d_{\mathbb M}}4}\end{align} with probability at least $1-\delta$; here, the last step is due to the forced equality in \cref{01misc1}.\\

For any function $\phi\in\operatorname{Lip}_{1,0}(\mathbb M)$, define ${\tilde {\phi}}_t:\mathbb M\rightarrow\mathbb R$ to be \begin{equation*}{\tilde {\phi}}_t(\boldsymbol{x})={\frac 1{|{\hat S}^\ell|}}\sum_{\boldsymbol{p}\in {\hat S}^\ell \boldsymbol{p}_0}\phi(\boldsymbol{p})\operatorname{H}_{\boldsymbol{p}}(\boldsymbol{x},t)\end{equation*} From stochastic completeness of $\mathbb M$, it follows that $\int_{\mathbb M}\operatorname{H}_{\boldsymbol{p}}(\boldsymbol{x},t)~d\sigma^{\mathbb M}(\boldsymbol{x})=1$; hence, putting $t=\epsilon^2$, one has \begin{align}\label{secondl}\int_{\mathbb M}{\tilde{\phi}}_{\epsilon^2}(\boldsymbol{x})~d\sigma^{\mathbb M}(\boldsymbol{x}) & = {\frac 1{|{\hat S}^\ell|}}\sum_{\boldsymbol{p}\in {\hat S}^\ell \boldsymbol{p}_0}\phi(\boldsymbol{p})\int_{\mathbb M}\operatorname{H}_{\boldsymbol{p}}(\boldsymbol{x},t)~d\sigma^{\mathbb M}(\boldsymbol{x}) \nonumber\\ &=\int_{\mathbb M}\phi(\boldsymbol{x})~d\nu(\boldsymbol{x})\,,\end{align} for any $\phi\in \operatorname{L}^1(\sigma^{\mathbb M})$. With $t=\epsilon^2$, this implies \begin{align}\label{last} W_1(\nu^\ast_{\epsilon^2},\nu)&=~\begin{vmatrix}\int_{\mathbb M}\phi(\boldsymbol{x})~d\nu^\ast_t(\boldsymbol{x})-\int_{\mathbb M} {\tilde{\phi}}_t(\boldsymbol{x})~d\sigma^{\mathbb M}(\boldsymbol{x})\end{vmatrix}\nonumber\\ &=~|\hat{\mathcal S}|^{-\ell} \begin{vmatrix}\sum_{\boldsymbol{p}\in {\hat{\mathcal S}}^\ell\boldsymbol{p}_0}\int_{\mathbb M}\begin{pmatrix}\phi(\boldsymbol{x})-\phi(\boldsymbol{p})\end{pmatrix}\operatorname{H}_{\boldsymbol{p}}(\boldsymbol{x},t)~d\sigma^{\mathbb M}(\boldsymbol{x})\end{vmatrix}\nonumber\\ &\leq~ |\hat{\mathcal S}|^{-\ell} \sum_{\boldsymbol{p}\in {\hat{\mathcal S}}^\ell\boldsymbol{p}_0}\int_{\mathbb M}\begin{vmatrix}\phi(\boldsymbol{x})- \phi(\boldsymbol{p})\end{vmatrix}\operatorname{H}_{\boldsymbol{p}}(\boldsymbol{x},t)~d\sigma^{\mathbb M}(\boldsymbol{x})\nonumber\\ &\leq~ |\hat{\mathcal S}|^{-\ell} \sum_{\boldsymbol{p}\in {\hat{\mathcal S}}^\ell\boldsymbol{p}_0}\int_{\mathbb M}\partial_{\mathbb M}(\boldsymbol{p},\boldsymbol{x})~\operatorname{H}_{\boldsymbol{p}}(\boldsymbol{x},t)~ d\sigma^{\mathbb M}(\boldsymbol{x})\nonumber\\ &\leq~ \sqrt {d_{\mathbb M}}\epsilon\end{align} by \cref{vari} and H\"{o}lder inequality applied to $\partial_{\mathbb M}(\boldsymbol{p},\boldsymbol{x})=\partial_{\mathbb M}(\boldsymbol{p},\boldsymbol{x})\cdot 1_{\mathbb M}$ while integrating with respect to $\operatorname{H}_{\boldsymbol{p}}(\boldsymbol{x},t)~d\sigma^{\mathbb M}(\boldsymbol{x})$. Therefore, for $t=\epsilon^2>0$ sufficiently small, equations \cref{1stw}, \cref{secondl}, and \cref{last} yield \begin{align*}W_1(\sigma^{\mathbb M},\nu)&\leq W_1(\sigma^{\mathbb M},\nu^\ast_t)+W_1(\nu^\ast_t,\nu)\\ &\leq \epsilon^{\frac{d_{\mathbb M}}2}d_{\mathbb M}^{-\frac{d_{\mathbb M}}4}+\sqrt d_{\mathbb M}\epsilon\\ &\leq 2\sqrt{d_{\mathbb M}\epsilon}\,.\end{align*} This completes the proof.
\end{proof}

\begin{remark}\label{0arem}
It is immediate from the foregoing proof --- in association with \cref{0rem01} --- 
that the upper-bound in \cref{Wassers} improves to $2\sqrt{d_{\mathbb M}}\epsilon$ 
if we ignore the low-dimensional cases $d_{\mathbb M}\leq 6$.
\end{remark}

\medskip

\subsection{Persistent Homology}\label{ssec:4}
In this subsection, we assume that we have oracle access to pairwise geodesic distance in $\mathbb M$.

\begin{definition}
A family \begin{displaymath}\{\chi_x^y: \mathbb S_{\boldsymbol{x}}\rightarrow \mathbb S_y\}_{0\leq x\leq y}\end{displaymath} of morphisms of finite simplicial complexes is said to be persistent if the following holds for all $0\leq x\leq y\leq z$: \begin{displaymath}\chi_y^z\circ \chi_x^y=\chi_x^z\end{displaymath}
\end{definition}

\begin{remark}
In functorial terms, let $\operatorname{FSC}$ be the category of finite simplicial complexes and morphisms. A persistent family is then a functor from the poset $(\mathbb R,\leq)$ to $\operatorname{FSC}$.
\end{remark}

\begin{definition}
 For a persistent family \begin{displaymath}\{\chi_{\boldsymbol{x}}^y: \mathbb S_{\boldsymbol{x}}\rightarrow \mathbb S_y\}_{0\leq x\leq y}\,,\end{displaymath} the $q$-th persistent $p$-homology of $\mathbb S_{\boldsymbol{x}}$, denoted $\operatorname{PH}_{p,q}(\mathbb S_{\boldsymbol{x}})$, is the image of the natural map \begin{displaymath}(\chi_{\boldsymbol{x}}^{x+q})_\star:\operatorname{H}_p(\mathbb S_{\boldsymbol{x}})\rightarrow \operatorname{H}_p(\mathbb S_{x+q})\,.\end{displaymath} Further, when $\{\chi_{\boldsymbol{x}}^{\boldsymbol{y}}:\mathbb S_{x}\rightarrow \mathbb S_{y}\}_{0\leq x\leq y}$ is the Vietoris-Rips complex, then the groups \begin{displaymath}\operatorname{PH}_p(\mathbb M):=\chi_{\boldsymbol{x}}^{x+q})_\star(\operatorname{H}_p(\mathbb S_{\boldsymbol{x}})\end{displaymath} are called the persistent homology of the complex.
\end{definition}

\begin{definition}\label{VR}
Let ${\mathcal S}\subseteq \mathbb K$ be a random subset consisting of independent Haar samples from $\mathbb K$, and let $\hat{{\mathcal S}}:={\mathcal S}\sqcup {\mathcal S}^{-1}$. For any integer $\ell>0$, consider the subset $\mathscr S_\ell:=\hat{{\mathcal S}}^\ell\boldsymbol{p}_0\subseteq\mathbb M$ constructed in the preceding section. Once and for all, fix an ordering of the points in each $\mathscr S_\ell$ that is compatible across the range of $\ell$. For any $x\geq 0$, let $\mathscr S_{\ell,x}$ be the Vietoris-Rips complex with vertex set $\mathscr S_{\ell}$ and radius $x$. Thus, a $t$-simplex in $\mathscr S_{\ell,x}$ is a subset $\{\boldsymbol{m}_0,\dots,\boldsymbol{m}_t\}\subseteq\mathscr S_{\ell}$ such that $\partial_{\mathbb M}(\boldsymbol{m}_{j_1},\boldsymbol{m}_{j_2})< x$ for every $j_1,j_2\in \{0,1,\dots,t\}$. For each $0\leq x\leq y$, let $\chi_\ell^{x,y}:\mathscr S_{\ell,x}\rightarrow \mathscr S_{\ell,y}$ be the natural inclusion of simplicial complexes. Thus, the morphisms \begin{displaymath}\chi_\ell^{x,y}:\mathscr S_{\ell,x}\rightarrow \mathscr S_{\ell,y},~0\leq x\leq y\end{displaymath} form a persistent family of finite simplicial complexes.
\end{definition}

For compact connected Riemannian symmetric spaces of dimension $d_{\mathbb M}$, the following result follows via standard arguments.


\begin{lemma}\label{somelemma}
Let $\tau>0$ be the (infimum of the) injectivity radius of $\mathbb M$. Suppose that 
${\mathcal S}\subset \mathbb K$ is a nonempty finite multisubset whose 
elements are selected independently at random from the Haar measure 
on $\mathbb K$. Let $\delta\in (0,\frac 12)$, and assume that $
\epsilon\in (0,2^{-e})$ is sufficiently small. Suppose that 
the cardinality of ${\mathcal S}$ satisfies \begin{equation}\label{0110seq1}
|{\mathcal S}|= 16\ln 2\left(\ln C_{\mathbb M} 
+\frac {d_{\mathbb M}}4\log_2\begin{pmatrix}
{\frac{d_{\mathbb M}}{2\epsilon^2}}\end{pmatrix}
+\frac{d_{\mathbb M}}2\log_2\log_2\begin{pmatrix}{\frac{d_{\mathbb M}}
{2\epsilon^2}}\end{pmatrix}+\ln \frac 1{e\delta}\right)\,.\end{equation} 
Let $\ell>0$ be an integer satisfying \begin{equation}\label{0110seq01}
\ell\geq \begin{pmatrix}d_{\mathbb M}-\frac{\bar{d}_{\mathbb M}+1}2\end{pmatrix}
\log_2\frac 1\epsilon+\log_2\begin{pmatrix}\frac{6C_{\mathbb M}}
{{\mathfrak v}_{\mathbb M}}\end{pmatrix}+\frac {d_{\mathbb M}}4\log_2
\begin{pmatrix}\frac 1{\pi r^2_{\epsilon,\mathbb M}}\end{pmatrix}
+\frac 12\log_2\Gamma\begin{pmatrix}{\frac {d_{\mathbb M}}2}+1
\end{pmatrix}\,,\end{equation} 
where \begin{displaymath}r_{\epsilon,\mathbb M}=2\epsilon\sqrt{\ln {\frac {3C_{\mathbb M}}{\epsilon^{2d-1}}}}\end{displaymath} satisfies $r_{\epsilon,\mathbb M}<4^{-1}\tau$; then, with probability at least $1-2\delta$, the geometric realization of $\mathscr S_{\ell,r_{\epsilon,\mathbb M}}$ --- as in \cref{VR} --- is homotopy equivalent to $\mathbb M$.
\end{lemma}

\begin{proof}
We assume that the random subset $\mathscr S_{\ell}\subseteq \mathbb M$ is an $r_{\epsilon,\mathbb M}$-cover of $\mathbb M$; by \cref{mainsym1}, this assumption is valid with probability at least $1-2\delta$. The families \begin{displaymath}{\mathscr O}_{r_{\epsilon,\mathbb M}}:=\{\operatorname{B}_{r_{\epsilon,\mathbb M}}(\boldsymbol{m}): \boldsymbol{m}\in\mathscr S_{\ell}\},~~{\mathscr O}_{2r_{\epsilon,\mathbb M}}:=\{\operatorname{B}_{2r_{\epsilon,\mathbb M}}(\boldsymbol{m}): \boldsymbol{m}\in\mathscr S_{\ell}\}\end{displaymath} of open geodesic balls --- of radius $r_{\epsilon,\mathbb M}\in (0,\frac\tau4)$ and $2r_{\epsilon,\mathbb M}\in (0,\frac\tau2)$, respectively --- with centers at the points in $\mathscr S_{\ell}$ forms an open cover of $\mathbb M$; moreover, the guarantee that $r_{\epsilon,\mathbb M}\in (0,\frac\tau4)$ ensures that each geodesic ball in these covers is diffeomorphic --- via the exponential map --- to an open Euclidean ball, and finite intersections of such geodesic balls are diffeomorphic to star-shaped open subsets in Euclidean space. Therefore, by nerve lemma, the \'{C}ech nerves ${\mathscr C}_{r_{\epsilon,\mathbb M}}$ and ${\mathscr C}_{2r_{\epsilon,\mathbb M}}$ --- based on ${\mathscr O}_{r_{\epsilon,\mathbb M}}$ and ${\mathscr O}_{2r_{\epsilon,\mathbb M}}$, respectively --- have the same homotopy type as $\mathbb M$. Since the Vietoris-Rips nerve $\mathscr S_{r_{\epsilon,\mathbb M}}$ interlaces the \'{C}ech nerves as \begin{displaymath}{\mathscr C}_{r_{\epsilon,\mathbb M}}\subseteq\mathscr S_{\ell,r_{\epsilon,\mathbb M}}\subseteq {\mathscr C}_{2r_{\epsilon,\mathbb M}}\,,\end{displaymath} the lemma follows.\end{proof}

\begin{remark}
Via application of the Vietoris-Rips complex and geometric realization functor, the topological invariants of the abstract Riemannian manifold $\mathbb M$ are identified to those of a subspace of an Euclidean space (albeit of a humongous dimension).
\end{remark}

Recall (see \cite{MR3230014}) that, for a compact metric measure space $(X,d_X,\mu_X)$ with metric $d_X$ and measure $\mu_X$, one defines \begin{displaymath}\Phi^{q,n}_{X,d_X,\mu_X}:=
(\operatorname{PH}_q)_\star(\mu_X^{\otimes n})\,,\end{displaymath} which is a measure in the \textit{Barcode} space $\mathcal B$. Let $\mathbb M$ be a compact connected Riemannian symmetric space of dimension $d$, Riemannian metric $\partial_{\mathbb M}$, and Haar-induced $\mathbb K$-invariant probability measure $\sigma^{\mathbb M}$. Let $\mathscr S_{\ell}:=\hat{\mathcal S}^\ell\boldsymbol{p}_0$, and $\sigma^{\mathbb M}_\ell$ the empirical probability measure supported on $\mathscr S_{\ell}$. 

\begin{theorem}\label{persistent1}
The following inequality holds for all nonempty subset $\mathcal S\subseteq \mathbb K$ 
of isometries of $\mathbb M$, all integers $\ell\geq 1$ and $q\geq 0$, with $n=|\hat{\mathcal S}|^\ell$
: \begin{equation}\frac 1n \cdot d_{\operatorname{Pr}}\left(\Phi^{q,n}_{{\mathbb M},\partial_{\mathbb M},\sigma^{\mathbb M}}, \Phi^{q,n}_{{\mathscr S_{\ell}},\partial_{\mathbb M},\sigma^{\mathbb M}_\ell}\right)\leq \operatorname{W}_1^{\frac 12}(\sigma^{\mathbb M},\sigma^{\mathbb M}_\ell)\end{equation}
\end{theorem}

\begin{proof}
We follow the arguments as in \cite[theorem 5.2]{MR3230014}. Since Wasserstein distance induces compact topology on $\mathscr P(\mathbb M\times \mathbb M)$ --- the space of probability measures on $\mathbb M\times \mathbb M$, there is a coupling $\lambda\in \operatorname{\Pi}(\sigma^{\mathbb M},\sigma^{\mathbb M}_\ell)$ such that \begin{equation*}W_1(\sigma^{\mathbb M},\sigma^{\mathbb M}_\ell)=\operatornamewithlimits{\mathbb E}_{\lambda}\, [\partial_{\mathbb M}(\boldsymbol{m}_1,\boldsymbol{m}_2)]\end{equation*} Suppose that $W_1(\sigma^{\mathbb M},\sigma^{\mathbb M}_\ell)=\epsilon_{q,n,\ell}^2$ for some $\epsilon_{q,n,\ell}>0$; it follows that \begin{equation}\mathbb P_\lambda\, [\partial_{\mathbb M}(\boldsymbol{m}_1,\boldsymbol{m}_2)\geq\epsilon_{q,n,\ell}]\leq\epsilon_{q,n,\ell}\end{equation} Since the $n$-fold product $\lambda^{\otimes n}$ induces a probability measure on $\mathbb M^n\times \mathscr S_{\ell}^n$, the push-forward $\hat\lambda_n:=(\operatorname{PH}_q^{\otimes n},\operatorname{PH}_q^{\otimes n})_\star(\lambda^{\otimes n})$ is a coupling between the push-forward measures $(\operatorname{PH}_q^{\otimes n})_\star(\sigma^{\mathbb M})$ and $(\operatorname{PH}_q^{\otimes n})_\star(\sigma^{\mathbb M}_\ell)$. 
Given $n$ independent $\lambda$-samples \begin{displaymath}(\boldsymbol{m}_1,\boldsymbol{s}_1),\cdots,(\boldsymbol{m}_n,\boldsymbol{s}_n),\end{displaymath} an application of triangle inequality implies \begin{equation*}|\partial_{\mathbb M}(\boldsymbol{m}_j,\boldsymbol{m}_{j'})-\partial_{\mathbb M}(\boldsymbol{s}_j,\boldsymbol{s}_{j'})|\leq \partial_{\mathbb M}(\boldsymbol{m}_j,\boldsymbol{s}_j)+\partial_{\mathbb M}(\boldsymbol{m}_{j'},\boldsymbol{s}_{j'})\,\end{equation*} and this further yields \begin{align*}
\mathbb P\left[\sup_{j,j'\in [n]}\left(\partial_{\mathbb M}(\boldsymbol{m}_j,\boldsymbol{s}_j)+\partial_{\mathbb M}(\boldsymbol{m}_{j'},\boldsymbol{s}_{j'})\right)\geq 2\epsilon_{q,n,\ell}\right]& \leq \mathbb P\left[\sup_{j\in [n]}\partial_{\mathbb M}(\boldsymbol{m}_j,\boldsymbol{s}_j)\geq \epsilon_{q,n,\ell}\right]\\ & \leq1-(1-\epsilon_{q,n,\ell})^n\\ & \leq n\epsilon_{q,n,\ell}\end{align*} By \emph{stability theorem} of Chazal et al (see \cite{MR3275299}), we have \begin{displaymath}\mathbb P\left[d_{\mathcal B}\left(\operatorname{PH}_q(\{\boldsymbol{m}_j\}), \operatorname{PH}_q(\{\boldsymbol{s}_j\})\right)\geq n\epsilon_{q,n,\ell}\right]\leq n\epsilon_{q,n,\ell}\,,\end{displaymath} which concludes the proof.
\end{proof}

Together with \cref{Wassers}, this yields the following immediate corollary: 

\begin{corollary}\label{persistent2}
Suppose that the nonempty subset $\mathcal S\subseteq\mathbb K$ is as in \cref{somelemma}, and that the integer $\ell\geq 0$ satisfies the inequality in \cref{somelemma}; then, the following inequality holds for all integers $q\geq 0$, with probability at least $1-\delta$: \begin{equation}\frac 1n \cdot d_{\operatorname{Pr}}\left(\Phi^{q,n}_{{\mathbb M},\partial_{\mathbb M},\sigma^{\mathbb M}}, \Phi^{q,n}_{{\mathbb S_{k,\ell}},\partial_{\mathbb M},\sigma^{\mathbb M}_\ell}\right)\leq \sqrt[4]{2d_{\mathbb M}\epsilon}\,.\end{equation}
\end{corollary}

\begin{remark}
It is immediate from the foregoing proof --- in association with \cref{0arem} --- 
that the upper-bound in \cref{persistent2} improves to $\sqrt[4]{2d_{\mathbb M}\epsilon^2}$ 
if we ignore the low-dimensional cases $d_{\mathbb M}\leq 6$.
\end{remark}

\bigskip

\medskip

\section{Conclusions}
\label{sec:conclusions}

On a compact Riemannian symmetric space $\mathbb M=\mathbb K/\mathbb H$ of 
dimension $d_{\mathbb M}$ and antipodal dimension $\bar{d}_{\mathbb M}$, 
a random Markov Chain --- whose transitions correspond to 
isometries selected independently at random from the Haar measure 
of the group of isometries of $\mathbb M$ --- has been explored 
and the finite time behaviour of the chain has been established. 
Explicitly, the first of the four main results states that for $|\hat{\mathcal S}|=
O_{\mathbb M}(d_{\mathbb M}\ln\frac{1}{\epsilon} + \ln \frac{1}{\delta})$ random isometries and 
$\ell = {O}(d_{\mathbb M} \ln 1/\epsilon)$, if one takes the image $\mathscr S$
of any fixed (but arbitrary) point on the symmetric 
space under all possible composition of length $\ell$ in the $|\hat{\mathcal S}|$-length random 
inverse-symmetric subset ${\mathcal S}\subseteq\mathbb K$, one obtains an $r_{\epsilon,\mathbb M}$-cover 
with high probability. The value of $|\hat{\mathcal S}|^\ell$ obtained here is close to 
the volumetric lower bound of $(1/\epsilon)^{\Omega(d_{\mathbb M})}$ on the size of a 
hypothetical optimum $\epsilon$-cover of $\mathbb M$. 
Secondly, we show that this $\epsilon$-cover is 
equidistributed with probability at least $1 - 2\delta$, in the sense that the 
$1$-Wasserstein distance of the uniform empirical measure supported on the cover is within 
$\sqrt{d_{\mathbb M}\epsilon}$ of the uniform measure on $\mathbb M$. 
Next, we show that such random subset $\mathscr S\subseteq\mathbb M$ 
is an approximate $\lambda_r$-design, for eigenvalue $\lambda_r$ of 
the Laplace-Beltrami on $\mathbb M$, provided $|\hat{\mathcal S}|
=O(\ln 1/\delta+d_{\mathbb M}\ln \lambda_{r_{\epsilon,\mathbb M}})$, 
and $\ell$ as above. Finally, we show that the Prokhorov distance between the persistent push-forwards 
of the uniform measure and the empirical measure supported on the random $\epsilon$-cover 
is at most $\sqrt[4]{d_{\mathbb M}\epsilon}$.\\

These results can respectively be applied to (1) approximately minimize any given $1$-
Lipschitz function on the symmetric space via a locally constant function by 
considering the Voronoi-cells of the space induced by the $\epsilon$-cover and 
evaluating the function on the $\epsilon$-cover, (2) 
to approximately integrate a $1$-Lipschitz function on the symmetric space, (3) 
to approximately integrate a smooth function in the subspace $\mathcal E_{\lambda_r}$, 
and (4) to approximate the persistence homology of a data cloud 
sampled from the uniform measure, with that of the constructed $\epsilon$-cover. In the first two cases 
the approximation is within an additive $\epsilon$ of the true value, whereas 
the last two cases has an error of order $\sqrt{\epsilon}$ (except in dimension $\leq 6$).
\bigskip

\section*{Acknowledgments}
The author is grateful to Hariharan Narayanan for insightful suggestions. 
We acknowledge partial funding from Deutsche 
Forschungsgemeinschaft (DFG, Deutschland) through grant 451920280.

\bibliographystyle{amsplain}
\def\noopsort#1{}\def\MR#1{}
\providecommand{\bysame}{\leavevmode\hbox to3em{\hrulefill}\thinspace}
\providecommand{\MR}{\relax\ifhmode\unskip\space\fi MR }
\providecommand{\MRhref}[2]{%
  \href{http://www.ams.org/mathscinet-getitem?mr=#1}{#2}
}
\providecommand{\href}[2]{#2}

\end{document}